\documentclass[12pt,letterpaper,reqno]{amsart}
\usepackage{graphicx}
\usepackage{amsmath}
\usepackage{amssymb}
\usepackage{amsfonts}
\usepackage{amsrefs}
\usepackage{amsthm}
\usepackage{enumerate}
\usepackage{hyperref}
\usepackage{cancel}
\usepackage[mathscr]{eucal}

\usepackage[usenames]{xcolor}

\newcommand{\R}{\mathbb{R}}

\newcommand{\Z}{\mathbb{Z}}
\newcommand{\ol}{\overline}
\newcommand{\cH}{{\mathcal H}}

\newcommand{\M}{\mathbf{M}}
\newcommand{\F}{\mathcal{F}}
\newcommand{\lebmeas}{{\mathcal{L}}}

\newtheorem{thm}{Theorem}
\newtheorem{lemma}[thm]{Lemma}
\newtheorem{prop}[thm]{Proposition}

\newtheorem{ex}[thm]{Example}

\theoremstyle{definition}
\newtheorem{definition}[thm]{Definition}
\newtheorem{assum}[thm]{Assumptions}
\newtheorem{remark}[thm]{Remark}

\newtheorem*{ack}{Acknowledgments}
\newtheorem*{outline}{Outline}

\DeclareMathOperator{\capac}{cap}

\DeclareMathOperator{\spt}{spt}
\DeclareMathOperator{\vol}{vol}
\DeclareMathOperator{\diam}{diam}
\DeclareMathOperator{\set}{set}
\DeclareMathOperator{\Lip}{Lip}
\DeclareMathOperator{\Lipb}{Lip^b}
\DeclareMathOperator{\LipB}{Lip_B}
\DeclareMathOperator{\Liploc}{Lip_{loc}}
\DeclareMathOperator{\Tan}{Tan^{(n)}}

\newcommand{\Mlocm}{\mathbf{M}_{\text{loc}, m}}
\newcommand{\Mm}{\mathbf{M}_{m}}
\newcommand{\Mlocmone}{\mathbf{M}_{\text{loc}, m-1}}
\newcommand{\Ilocm}{\mathbf{I}_{\text{loc}, m}}

\newcommand{\Ilocmplusone}{\mathbf{I}_{\text{loc}, m+1}}
\renewcommand{\Im}{\mathbf{I}_{m}}
\newcommand{\In}{\mathbf{I}_{n}}
\newcommand{\Imone}{\mathbf{I}_{m+1}}
\newcommand{\scrIlocm}{\mathcal{I}_{\text{loc},m}}
\newcommand{\scrIlocmone}{\mathcal{I}_{\text{loc},m-1}}

\newcommand{\VF}{\mathcal{VF}}
\newcommand{\toF}{\xrightarrow{\F}}
\newcommand{\toVF}{\xrightarrow{\VF}}

\begin{document}

\title{Semicontinuity of capacity under pointed intrinsic flat convergence}
\author{Jeffrey L. Jauregui}
\address{Dept. of Mathematics,
Union College, 807 Union St.,
Schenectady, NY 12308,
United States}
\email{jaureguj@union.edu}
\author{Raquel Perales}
\address{CONACyT Research Fellow at the Math Institute of the 
National Autonomous University of Mexico, Oaxaca. Mexico}
\email{raquel.perales@im.unam.mx}
\author{Jacobus W. Portegies}
\address{Dept. of Mathematics and Computer Science,
Eindhoven University of Technology,
De Zaale, Eindhoven,
The Netherlands}
\email{j.w.portegies@tue.nl}
\date{\today}

\begin{abstract}
The concept of the capacity of a compact set in $\R^n$ generalizes readily to noncompact Riemannian manifolds and, with more substantial work, to metric spaces (where multiple natural definitions of capacity are possible). Motivated by analytic and geometric considerations, and in particular
Jauregui's definition of capacity-volume mass and Jauregui and Lee's results on the lower semicontinuity of the ADM mass and Huisken's isoperimetric mass,
we investigate how the capacity functional behaves when the background spaces vary.  Specifically, we allow the background spaces to consist of a sequence of local integral current spaces converging in the pointed Sormani--Wenger intrinsic flat sense. For the case of volume-preserving ($\VF$) convergence, we prove two theorems that demonstrate an upper semicontinuity phenomenon for the capacity: one version is for balls of a fixed radius centered about converging points; the other is for Lipschitz sublevel sets. Our approach is motivated by Portegies' investigation of the semicontinuity of eigenvalues under $\VF$ convergence. We include examples to show the semicontinuity may be strict, and that the volume-preserving hypothesis is necessary. Finally, there is a discussion on how capacity and our results may be used towards understanding the general relativistic total mass in non-smooth settings.
\end{abstract}

\maketitle 

\section{Introduction}
The capacity of a compact set $K \subset \R^n$, $n \geq 3$, is defined as:
$$\capac(K) = \inf_{\phi} \left\{\frac{1}{(n-2)\omega_{n-1}}  \int_{\R^n} |\nabla \phi|^2 dV\; :\; \phi \text{ is Lipschitz with compact support, and } \phi \equiv 1 \text{ on } K\right\},$$
where $\omega_{n-1}$ is the hypersurface area of the unit $(n-1)$-sphere. This notion is also called the harmonic, electrostatic, or Newtonian capacity.
If $\partial K$ is sufficiently regular (e.g., a $C^1$ hypersurface), then there exists a unique harmonic function $u$ on $\R^n \setminus K$, equalling 1 on $\partial K$ and approaching 0 at infinity, such that
$$\capac(K) =  \frac{1}{(n-2)\omega_{n-1}} \int_{\R^n \setminus K} |\nabla u|^2 dV = - \frac{1}{(n-2)\omega_{n-1}} \int_{S} \frac{\partial u}{\partial \nu} dA,$$
for any surface $S$ enclosing $\partial K$. For example, a ball of radius $r$ has capacity equal to $r^{n-2}$. Capacity is monotone under set inclusion and enjoys nice measure-theoretic properties, such as inner and outer regularity \cite{EG}. Geometrically, it can be bounded below by the volume radius $\left(\frac{n\vol(K)}{\omega_{n-1}}\right)^{\frac{1}{n}}$ of $K$ (due to Poincar\'e--Faber--Szeg\"o) and, if $\partial K$ is convex, bounded above in terms of the total mean curvature of $\partial K$ (due to Szeg\"o) \cite{PS}.

Capacity also makes sense with an analogous definition in complete Riemannian manifolds, such as asymptotically flat manifolds. (Without some control on the asymptotics, however, the capacity could be zero for every compact set.) It is natural to ask how capacity behaves along a converging sequence of Riemannian manifolds. For example, it is not difficult to show that if $M$ is a smooth manifold equipped with a sequence of complete Riemannian metrics $\{g_i\}$, and $g_i$ converges uniformly (i.e, in $C^0$) on compact sets to a Riemannian metric $g$, then for any compact set $K \subset M$,
$$\limsup_{i \to \infty} \capac_{g_i}(K) \leq \capac_g(K).$$
More generally, an analogous statement holds for a sequence of complete Riemannian manifolds converging in the pointed $C^0$ Cheeger--Gromov sense (see Appendix A). In fact, strict inequality can hold (see Section \ref{sec_examples}). 
This upper semicontinuity of capacity contrasts sharply with two other natural notions of the ``size'' of $K$, the volume and perimeter, which are of course continuous under $C^0$ convergence of the background metrics. This different behavior of capacity in the case of $C^0$ convergence springs from its non-local nature, specifically its dependence on the geometry ``at infinity.''

The aim of this paper is to study capacity in lower regularity background spaces and in particular to analyze its behavior under lower regularity convergence. While capacity in lower regularity (such as in metric measure spaces) has already received significant attention in the analysis literature (see below), we believe the study of its continuity is novel. We are also interested in capacity and its continuity properties for geometric reasons. For example a recent paper by Jauregui \cite{Jau} suggests a definition of total mass in general relativity for asymptotically flat 3-manifolds that is based on the capacity--volume inequality, generalizes the well-known ADM mass, and is inspired by Huisken's isoperimetric mass \cites{Hui1,Hui2}. Several important open problems in general relativity related to the total mass seem to naturally involve Sormani--Wenger intrinsic flat (``$\F$'') limits (we refer the reader to \cite{Sor2} and \cite{JL}, for example), so the behavior of capacity under such convergence is of interest. For instance, Jauregui and Lee showed lower semicontinuity of Huisken's isoperimetric mass \cites{Hui1,Hui2} under pointed $\VF$-convergence 
as in 
\cite{JL}, cf. 
 Definition 
 \ref{def_pointed_F}
 (where ``$\VF$'' refers to volume-preserving intrinsic flat convergence). The definition of mass in \cite{Jau} involves capacity with a negative sign, so the upper semicontinuity we prove here is supportive of lower semicontinuity of that mass. We continue this discussion in Section \ref{sec_mass}.

Our approach to establishing the upper semicontinuity of capacity is inspired by Portegies' proof that certain min-max values of the Laplacian on a compact Riemannian manifold $M$ are upper semicontinuous under $\VF$ convergence \cite{Por}. Recall that, for example, the first such eigenvalue,
$$\lambda_1 = \inf_{f \in C^{\infty}(M)} \left\{\int_M |\nabla f|^2 dV \; : \; \int_M f^2 dV=1, \int_M fdV = 0 \right\},$$
varies continuously with respect to $C^2$ convergence of  Riemannian metrics, but may only be upper semicontinuous for weaker types of convergence such as measured Gromov--Hausdorff convergence \cite{Fuk}. Since capacity is only interesting in noncompact spaces, we will specifically study the behavior of capacity under \emph{pointed} $\VF$-convergence.

 In Section \ref{sec_examples}, we provide examples to illustrate that there are essentially two distinct reasons for the capacity to jump under a $\VF$-limit, and both jumps ``go the same way.'' One reason is non-uniform control at infinity, which may be seen even under smooth convergence; the other is an effect of the relatively coarse nature of $\VF$ convergence.

Below we state one of our main theorems, that the capacity of closed balls of a fixed radius about converging points in a sequence of converging spaces cannot jump down in a limit. The natural setting for intrinsic flat convergence is the integral current spaces of Sormani and Wenger \cite{SW}, which are constructed using the integral currents on metric spaces of Ambrosio and Kirchheim \cite{AK_cur}. We will use local versions of these spaces (building on Lang and Wenger's locally integral currents \cite{LW}) and pointed convergence, and make use of the definition of Dirichlet energy in these spaces appearing in \cite{Por} to define capacity. The relevant definitions will be given in Section \ref{sec_background}.

\begin{thm}
\label{thm_balls_intro}
Let $N_i = (X_i, d_i, T_i)$ and $N_\infty = (X_\infty, d_\infty, T_\infty)$ be local integral current spaces of dimension $m \geq 2$, such that $N_i \to N_\infty$ in the pointed volume-preserving intrinsic flat sense with respect to $p_i \in X_i$ and $p_\infty \in  X_\infty$. Suppose the closed ball $\ol B(p_\infty,r')$ in $X$ is compact for some $r'>0$. Then for all $0<r<r'$,
\begin{equation}
\label{eqn_limsup_balls_intro}
\limsup_{i \to \infty} \capac_{N_i} (\ol B(p_i,r)) \leq \capac_{N_\infty}(\ol B(p_\infty,r)).
\end{equation}
\end{thm}
The other main theorem, Theorem \ref{thm_sublevel}, is a version of Theorem \ref{thm_balls_intro} that replaces closed balls with sublevel sets of Lipschitz functions. Both of these theorems will be proved using the technical result Theorem \ref{thm_extrinsic}, an extrinsic version of the theorems which itself is inspired by Theorem 6.2 in \cite{Por}. 

Note that integral current spaces are metric spaces $(X,d)$ with a current structure, $T$, that provides a countable collection of bi-Lipschitz charts with integer weights as well as a notion of boundary which also has bi-Lipschitz charts with integer weights. The mass measure of $T$, denoted $\|T\|$, allows one to view an integral current space as a metric measure space $(X,d,\|T\|)$. See Section \ref{sec_background} for details. Note that convergence of integral current spaces in the $\VF$ sense implies convergence of the underlying metric measures spaces; see Lemma \ref{lemma_mass_convergence}.

We note that there is a body of literature on  capacities of sets in metric measure spaces. While we make no attempt at a comprehensive account, we discuss this partially here, referring the reader to the books by Bj\"orn and Bj\"orn \cite{BB} and  Heinonen, Koskela, Shanmugalingam, and Tyson \cite{HKST} for example. One approach  begins with the definition of Sobolev functions on a metric measure space, due to Haj\l asz \cite{Haj}, based on the Hardy--Littlewood maximal operator. Kinnunen and Martio used Haj\l asz's definition to develop a Sobolev $p$-capacity \cite{KM}. (Sobolev $p$-capacity, in contrast to the capacity we consider here, is found by minimizing the Sobolev norm, i.e., it includes the $L^p$ norm of the test functions.) A different approach is to consider, instead of Haj\l asz's Sobolev spaces, the Sobolev spaces defined by Shanmugalingam \cite{Sha}, called Newtonian spaces, an approach to Sobolev spaces using weak upper gradients. This is explored in detail in \cite{BB}, where the version of capacity based only on the $L^p$ norm of the weak upper gradient (as we consider in this paper) is called the variational capacity. Since the norm of the tangential differential is $\|T\|$-almost-everywhere equal to the minimal relaxed gradient (see Theorem \ref{thm_df} in Appendix B, cf. \cite[Theorem 5.2]{Por}), this latter capacity is equivalent to that which we use in this paper.
We continue this discussion in Appendix B and also refer the reader to \cite{GT} for additional discussion on capacities in metric measure spaces.

\begin{outline}
Section \ref{sec_background} covers the essential background material, including Ambrosio--Kirchheim currents on metric spaces, flat convergence, integral current spaces, and Sormani--Wenger intrinsic flat convergence, before moving on to local integral current spaces and pointed $\F$- and $\VF$-convergence. We also recall the definition of the Dirichlet energy of a Lipschitz function using the tangential differential (where some details are deferred to Appendix B), and use that to define the capacity in a local integral current space. The main results are presented and proved in Section \ref{sec_main}, and several examples are given in Section \ref{sec_examples} to demonstrate how capacity may ``jump up'' and show that  volume-preserving convergence is necessary. Section \ref{sec_mass} includes further discussion regarding how capacity on integral current spaces may be of interest for problems involving mass in general relativity.
\end{outline}

\begin{ack}
R. P. acknowledges support from CONACyT Ciencia de Frontera 2019  CF217392 grant. The authors would like to thank the referees for their careful read of the paper and their very helpful comments.
\end{ack}

\section{Definitions and basic objects}
\label{sec_background}

In 1960 Federer--Fleming \cite{FeFle} used de Rham's $m$-dimensional currents on $\R^n$  \cite{deRham}, extending the notion of $m$-dimensional submanifolds, to find limits of sequences of such submanifolds. Currents are functionals on differential forms $f d\pi_1 \wedge \dots \wedge d\pi_m$, and an embedded submanifold $\varphi: M^m \to \R^n$ can be viewed as a current as follows:
\begin{equation}
\label{current_integration}
\varphi_{\#}[[M]](f d\pi_1 \wedge \dots \wedge d\pi_m) = \int_{\varphi(M)}f d\pi_1 \wedge \dots \wedge d\pi_m.
\end{equation}
Federer--Fleming defined weak convergence of currents and extended Whitney's notion of a flat norm \cite{Wh57} to currents. 
They also defined rectifiable, integral, and normal currents. In 2000, Ambrosio--Kirchheim defined an $m$-dimensional current on a complete metric space $Z$ \cite{AK_cur} using de Giorgi's $(m+1)$-tuples of Lipschitz functions $(f,\pi_1, \dots, \pi_m)$ from \cite{deGiorgi}, viewing a submanifold $\varphi:M^m \to Z$ as a current acting on tuples as follows:
$$\varphi_{\#}[[M]](f, \pi_1, \dots, \pi_m) = \int_M (f \circ \varphi) d(\pi_1 \circ \varphi) \wedge \dots \wedge d(\pi_m \circ \varphi).$$
Wenger defined and studied the flat distance $d_F^Z$ between currents in $Z$ \cite{We2007}.

It is of interest in geometric analysis to study sequences of distinct Riemannian manifolds and metric spaces that do not live in a common extrinsic space. The well-known Gromov--Hausdorff distance between compact metric spaces is defined by
$$d_{GH} ((X_1, d_1), (X_2,d_2)) = \inf d_H^Z \Big(\phi_1(X_1), \phi_2(X_2)\Big),$$
where the infimum is taken over all metric spaces $Z$ and all distance-preserving maps $\phi_i:X_i \to Z$ of the Hausdorff distance $d_H^Z$. In 2011 Sormani--Wenger imitated this idea to define the intrinsic flat distance, applying Ambrosio--Kirchheim theory \cite{AK_cur} and Wenger's flat distance $d_F^Z$ between currents in $Z$ \cite{We2007}. They extended the notion of a compact oriented Riemannian manifold $(M^m,g)$ to an integral current space $(M,d_g,[[M]])$, where $[[M]]$ represents the current that ``integrates'' $(m+1)$-tuples of Lipschitz functions over $M$, and defined the intrinsic flat distance between integral current spaces $(X_i, d_i, T_i)$, $i=1,2$, of the same dimension as 
$$d_{\F}((X_1,d_1,T_1), (X_2, d_2, T_2)) = \inf d_F^Z \Big(\phi_{1\#} (T_1), \phi_{2\#} (T_2)\Big)$$
with the infimum taken over all complete metric spaces $Z$ and all distance-preserving maps $\phi_i:X_i \to Z$. In general an integral current space $(X,d,T)$ is a rectifiable metric space with integer weights and integer-weighted boundaries, endowed with an integral current structure, $T$, defined by the bi-Lipschitz charts of its rectifiable structure, that acts on tuples as in \cite{AK_cur}. The precise definition is recalled in Definition \ref{def_ICS}.

Our primary goal is to study the capacity of compact sets in an integral current space, and in particular the continuity behavior of capacity. We certainly want the theory to include ambient spaces of infinite mass (volume), such as $\R^n$, yet Sormani--Wenger integral current spaces (built using Ambrosio--Kirchheim currents) by definition have finite mass. To work around this, we will use Lang--Wenger's extension \cite{LW} of Ambrosio--Kirchheim \cite{AK_cur} integral currents on metric spaces to so-called \emph{locally integral currents}, then generalize Sormani--Wenger's definition of integral current space accordingly. We recall the details of locally integral currents in Section \ref{sec_currents}. In Section \ref{sec_flat_weak}, we recall local flat, weak, and flat convergence. Next, in Section \ref{sec_local_currents}, we recall the details of  Sormani--Wenger intrinsic flat convergence and its pointed version. Finally, in Section \ref{sec_dirichlet}, we give the definition of capacity in local integral current spaces. This definition involves a differential of Lipschitz functions on a metric space, and there are a number of ideas from \cite{AK_rect} that will be recalled in Appendix B. This approach is inspired by the work of Portegies on studying the behavior of eigenvalues of the Laplacian under $\VF$-convergence \cite{Por}.

\subsection{Locally integral currents}
\label{sec_currents}

The goal of this section is to arrive at the definition of a locally integral current. Below we are essentially summarizing the parts of \cite{LW} that will get us to that point, including only the minimal details, with no proofs. 

First, we recall metric functionals and how they produce an outer measure. Currents of locally finite mass will be those metric functionals whose measure behaves well from the inside and outside. Locally integer rectifiable currents are defined next, and then finally locally integral currents.

Given a metric space $Z$, define:
\begin{itemize}
\item $\Lip(Z)$ as the vector space of Lipschitz functions $Z \to \R$,
\item $\Liploc(Z)$ as the vector space of functions $Z \to \R$ that are Lipschitz on bounded sets,
\item $\Lipb(Z)$ as the vector space of Lipschitz functions $Z \to \R$ that are bounded, and
\item $\LipB(Z)$ as the vector space of Lipschitz functions $Z \to \R$ that are bounded with bounded support.
\end{itemize}
The Lipschitz constant of a function $f:Z \to \R$ will be denoted by $\Lip(f)$. 

For an integer $m \geq 0$, an $m$-dimensional metric functional $T$ will act on $(m+1)$-tuples of functions in $\LipB(Z) \times \left[\Liploc(Z)\right]^m$ and produce a real number. Such an $m$-tuple $(f, \pi_1, \ldots, \pi_m)$ (sometimes denoted more briefly by $(f,\pi)$) should be conceptually thought of as the differential form ``$f d\pi_1 \wedge \dots \wedge d\pi_m$'', so metric functionals will generalize the idea of currents. The precise definition of $T$ being a \emph{metric functional} is that:
\begin{enumerate}
\item[(i)] $T$ is multilinear.
\item[(ii)] (continuity) Suppose $f \in \LipB(Z)$ and $(\pi_1, \ldots, \pi_m) \in \Liploc(Z)^m$, and that we have $m$ sequences $\{\pi_i^j\}_{j=1}^\infty$ in $\Liploc(Z)$ such that $\pi_i^j \to \pi_i$ pointwise as $j \to \infty$ for each $i=1,\ldots, m$, and the Lipschitz constants of $\pi_i^j$ are uniformly bounded in $j$ on any bounded subset of $Z$. Then
$$\lim_{j \to \infty} T(f, \pi_1^j, \ldots, \pi_m^j) = T(f,\pi_1, \ldots, \pi_m).$$

\item[(iii)] (locality) Consider $(f, \pi_1, \ldots, \pi_m) \in \LipB(Z) \times \left[\Liploc(Z)\right]^m$. Suppose that one of the $\pi_i$ is constant on the $\delta$-neighborhood of $\spt(f)$ for some $\delta > 0$. Then
$$T(f, \pi_1, \ldots, \pi_m) = 0.$$
\end{enumerate}
Metric functionals as defined here are natural analogs of the metric functionals considered by Ambrosio and Kirchheim \cite{AK_cur} (in that case, $f$ was a bounded Lipschitz function and the $\pi_i$ were required to be (globally) Lipschitz).

If $X$ and $Y$ are metric spaces and $\varphi \in \Liploc(X,Y)$ with $\varphi^{-1}(A)$ bounded for any bounded set $A \subseteq Y$, a metric functional $T$ on $X$ can be pushed forward to a metric functional on $Y$ of the same dimension as follows:
$$(\varphi_{\#}T)(f, \pi_1, \ldots, \pi_m)  = T(f \circ \varphi, \pi_1 \circ \varphi, \ldots, \pi_m \circ \varphi).$$

The \emph{boundary} of an $m$-dimensional metric functional $T$, $m \geq 1$, is an $(m-1)$-dimensional metric functional $\partial T$ defined by
$$\partial T(f, \pi_1, \ldots, \pi_{m-1}) = T(\sigma, f, \pi_1, \ldots, \pi_{m-1}),$$
where $\sigma$ is any bounded Lipschitz function with bounded support that is identically 1 on the support of $f$. In \cite{LW} it is shown this definition is independent of the choice of $\sigma$ and defines a metric functional. The boundary satisfies nice properties, such as commuting with the push-forward and $\partial (\partial T)=0$. When $N$ is a complete Riemannian manifold with boundary, we can define $T=[[N]]$ by $T(f,\pi_1, \ldots, \pi_m) = \int_N f d\pi_1 \wedge \dots \wedge d\pi_m$, and in this case $\partial T = [[\partial N]]$ because 
$$\int_{\partial N} f d\pi_1 \wedge \dots \wedge d\pi_{m-1} =\int_N \sigma df \wedge d\pi_1 \wedge \dots \wedge d\pi_{m-1}.$$
Note this is only true for $f \in \Lip_B(N)$ and $\sigma$ as described above.

To define currents, we require the notion of the mass measure $\|T\|$ associated to a metric functional $T$. Ambrosio--Kirchheim's definition requires finite mass, but the approach we take following Lang--Wenger will require only locally finite mass.

Given an $m$-dimensional metric functional $T$ and an open set $V \subseteq Z$, the \emph{mass of $T$ in $V$} is defined to be
$$\M_V(T) = \sup \sum_{\lambda} T(f^\lambda, \pi^\lambda),$$
where the supremum is taken over all finite collections 
$f^\lambda \in \LipB(Z)$, 
$\pi^\lambda = (\pi_1^\lambda, \ldots, \pi_m^\lambda) \in \Lip(Z)^m$ 
such that $\Lip(\pi^\lambda_i) \leq 1$, $f^\lambda$ is supported in $V$, 
and $\sum_{\lambda} |f^\lambda| \leq 1$ everywhere. This 
definition serves as the mass of $T$ for any open set. For an arbitrary subset $A \subseteq Z$, we define
\begin{equation}
\label{mass_measure_def}
\|T\|(A) = \inf\{ \M_V(T) \; | \; V \supseteq A \text{ is open}\},
\end{equation}
so certainly $\|T\|(A) = \M_A(T)$ if $A$ is open.

\begin{definition}[\cite{LW}]
An \emph{$m$-dimensional metric current on $Z$ with locally finite mass}, $m \geq 0$, is an $m$-dimensional metric functional $T$ on $Z$ such that given any $\epsilon >0$ and any bounded open set $U \subseteq Z$, we have $\M_U(T) < \infty$ and that there exists a compact set $C \subseteq U$ such that  $\M_{U \setminus C}(T) < \epsilon$.
\end{definition}
The set of such objects will be denoted $\Mlocm(Z)$. For $T \in \Mlocm(Z)$, $\|T\|$ defined above is a Borel regular outer measure on $Z$ that is concentrated on a $\sigma$-compact set \cite[Proposition 2.4]{LW}. 

There is nice compatibility between Ambrosio--Kirchheim currents and the metric currents with locally finite mass; see \cite[Section 2.5]{LW}.
First, define:
$$\Mm(Z)= \{T \in \Mlocm(Z) \;:\; \|T\|(Z) < \infty\}.$$
In Section 2.5 of \cite{LW} it is shown that $\Mm(Z)$ may be identified with the set of currents on $Z$ as originally defined by Ambrosio and Kirchheim in \cite{AK_cur}, and the measures $\|T\|$ defined in \cite{AK_cur} and in \eqref{mass_measure_def} agree on Borel sets.
It is straightforward to verify also that the set of integral currents on $Z$ as defined in \cite{AK_cur} may be identified with
$$\Im(Z)= \{T \in \Ilocm(Z) \;:\; \|T\|(Z) < \infty\},$$
where the precise definition of locally integral currents $\Ilocm(Z)$ appears below in Definition \ref{def_local_integral_current}. (While currents and integral currents in \cite{AK_cur} were only defined on complete spaces, it is pointed out in Section 2.2 of \cite{LW} that the completeness restriction can be avoided.) We therefore may take the above equations as definitions of currents and integral currents on $Z$.

We also recall the definition of two ways of identifying where the $m$-dimensional metric functional $T$ ``lives'': the support and the canonical set:
\begin{align*}
\spt(T) &= \{z \in Z \; : \; \|T\|(B(z,r))> 0 \text{ for all } r > 0\}\\
\set(T) &= \{z \in Z \; : \; \liminf_{r \to 0} \frac{\|T\|(B(z,r))}{r^m} > 0\}.
\end{align*}
In \cite{AK_cur} $\set(T)$ is denoted by $S_T$, but we use the notation from \cite{SW}.

For example, any complete, connected, oriented Riemannian manifold $N^m$ gives rise to a $m$-dimensional metric functional with locally finite mass $T=[[N]]$, given as integration \cite[Section 2.8]{LW}:
\begin{equation}
\label{eqn_integration}
T(f,\pi_1, \ldots, \pi_m) = \int_N f d\pi_1 \wedge \dots \wedge d\pi _m.
\end{equation}
In this case, $\|T\|$ agrees with the Riemannian volume measure (see \cite[Example 2.32]{SW} and \cite[Lemma 22(a)]{JL}), and so $\spt(T)=\set(T)=M$.

Recall that metric functionals $T$ are defined as functions $\LipB(Z) \times \left[\Liploc(Z)\right]^m \to \R$. However, it is shown in \cite{LW} that if $T \in \Mlocm(Z)$, then $T$ has a natural extension to the larger space of $(m+1)$-tuples $(f,\pi)$ such that $f$ is a bounded Borel function with bounded support (and the $\pi$ are as before). This allows us to restrict $T \in \Mlocm(Z)$ to a bounded Borel set $A \subseteq Z$ as follows:

$$(T \llcorner A)(f,\pi) := T(f\chi_A, \pi).$$
Furthermore, $T \llcorner A \in \Mlocm(Z)$ and $\|T \llcorner A\| = \|T\| \llcorner A$.

Our main interest is in the definition of locally integral currents, so we build up to that now.

A subset $C \subseteq Z$ is a \emph{compact $m$-rectifiable set} if there exist compact sets $K_1, \ldots, K_n \subset \R^m$ and Lipschitz maps $f_i: K_ i \to Z$ such that
$$C = \bigcup_{i=1}^n f_i(K_i).$$

\begin{definition}[\cite{LW}]
An $m$-dimensional metric functional $T$ on $Z$ is a \emph{locally integer rectifiable current} if:
\begin{enumerate}[(a)]
\item Given $\epsilon >0$ and any bounded open subset $U$ of $Z$, there is a compact $m$-rectifiable set $C \subseteq Z$ such that $\M_U(T) < \infty$ and $\M_{U \setminus C}(T) < \epsilon$. 
\item For every bounded Borel set $B \subseteq Z$ and every Lipschitz map $\varphi: Z \to \R^m$, we have
$$\varphi_{\#} (T \llcorner B)(f, \pi) = \int_{\R^m} \theta f d\pi_1 \wedge \ldots \wedge d\pi_m,$$
for some integrable $\theta: \R^m \to \Z$. 
\end{enumerate}
The abelian group of locally integer rectifiable $m$-currents on $Z$ will be denoted by $\scrIlocm(Z)$.
\end{definition}
By (a) above, $\scrIlocm(Z) \subseteq \Mlocm(Z)$.

\subsection{Flat and weak convergence of integral currents}
\label{sec_flat_weak}

\begin{definition}[\cite{LW}]
\label{def_local_integral_current}
An $m$-dimensional locally integer rectifiable current $T$ is a \emph{locally integral current} if $m=0$, or, for $m \geq 1$,  if $\partial T \in \scrIlocmone(Z)$. (By \cite[Theorem 2.2]{LW}, it is sufficient to assume $\partial T \in \Mlocmone(Z)$.) The abelian group of locally integral currents on $Z$ will be denoted by $\Ilocm(Z)$.
\end{definition}

\begin{definition}[\cite{LW}]
\label{def_local_flat}
A sequence $T_i$ in $\Ilocm(Z)$ converges to $T \in \Ilocm(Z)$ in the \emph{local flat topology} if, for every closed, bounded set $B \subseteq Z$ there exists a sequence $S_i$ in $\Ilocmplusone(Z)$ such that
$$\|T - T_i - \partial S_i\|(B) + \|S_i\|(B) \to 0$$
as $i \to \infty$. Moreover, $T_i$ in $\Ilocm(Z)$ converges to $T \in \Ilocm(Z)$ \emph{weakly} if for any $f \in \LipB(Z)$ and $(\pi_1, \ldots, \pi_m) \in \Liploc(Z)^m$, we have
$$\lim_{i \to \infty} T_i(f, \pi_1, \ldots, \pi_m) = T(f,\pi_1, \ldots, \pi_m).$$
\end{definition}

Analogous to the corresponding statement for classical currents, if $T_i \to T$ in the local flat topology, then $T_i \to T$ weakly. 
Moreover, under weak convergence, $$\liminf_{i \to \infty} \|T_i\|(U) \geq \|T\|(U)$$ 
for all bounded open sets $U$. We also require the following lemma, generalizing \cite[Lemma 2.7]{Por} to the local case.

\begin{lemma}
\label{lemma_mass_convergence}
Suppose $T_i \to T$ weakly in $\Ilocm(Z)$, and suppose there exists a bounded open set $U \subseteq Z$ such that $\|T_i\|(U) \to \|T\|(U)$ as $i \to \infty$. Then $\|T_i\| \llcorner U \to \|T\| \llcorner U$ weakly as a bounded sequence of finite Borel measures. That is, for every bounded, continuous function $f:Z \to \R$ supported in $U$,
$$\int_Z f d\|T_i\| \to \int_Z f d\|T\|$$
as $i \to \infty$. In particular, for every closed set $C \subset U$,
$$\limsup_{i \to \infty} \|T_i\|(C) \leq \|T\|(C).$$
\end{lemma}
This follows immediately from the proof of \cite[Lemma 2.7]{Por}, which uses the portmanteau theorem.

We conclude this section by recalling Wenger's flat distance for currents of finite mass.  Given two integral $m$-currents $T_1$ and $T_2$ on $Z$, now taken to be a complete metric space, Wenger defined their \emph{flat distance} \cite{We2007}:
$$d_{F}^Z(T_1, T_2) = \inf_{A \in \Im(Z), B \in \Imone(Z)}\{\M(A) + \M(B) \; :\; T_2 - T_1 = A+\partial B\},$$
where we use $\M(A)$ to mean $\|A\|(Z)$, etc. This will be used in the definition of Sormani--Wenger intrinsic flat distance in the next section.

\subsection{Local integral current spaces and pointed $\F$-convergence}
\label{sec_local_currents}

We first recall Sormani--Wenger's definition of an integral current space:
\begin{definition} [\cite{SW}]
\label{def_ICS}
An \emph{integral current space} $N=(X,d,T)$ is a metric space $(X,d)$ equipped with an integral current $T$ on the completion $(\ol X, \ol d)$, such that $\set(T)=X$.
The \emph{dimension} of $N$ is the dimension of $T$.
\end{definition}

\begin{definition} [\cite{SW}]
The \emph{Sormani--Wenger intrinsic flat distance} between integral current spaces
$N_1=(X_1, d_1, T_1)$ and $N_2=(X_2, d_2, T_2)$ of dimension $m$ is:
$$d_{\F}(N_1, N_2) = \inf_{Z,\varphi_1, \varphi_2} \{d_F^Z(\varphi_{1\#}(T_1),\varphi_{2\#}(T_2))\},$$
where the infimum is taken over all complete metric spaces $Z$ and distance-preserving maps $\varphi_1:X_1 \to Z$ and $\varphi_2:X_2 \to Z$.
\end{definition}

Note that distance-preserving maps are also called isometric embeddings, but we refrain from using the latter terminology to avoid confusion with Riemannian isometric embeddings.

Note also that the $\varphi_i$ canonically extend to $\ol X_i$ as distance-preserving maps, so the push-forwards are well-defined.
 
In \cite{SW} it is shown that $d_{\F}$ defines a distance on the set of equivalence classes of precompact integral current spaces of dimension $m$ (where $d_{\F}(N_1, N_2) =0$ if and only if there exists an isometry $\varphi: X_1 \to X_2$ so that $\varphi_\#(T_1)=T_2$).

\begin{definition}[\cite{SW}, \cite{Sor2}]
A sequence of $m$-dimensional integral current spaces $N_i$ \emph{$\F$-converges} to an $m$-dimensional integral current space $N$ (written $N_i \toF N$) if
$$d_{\F}(N_i, N) \to 0$$
and \emph{$\VF$-converges} to $N$ (written $N_i \toVF N$) if
$$d_{\F}(N_i, N) + |\M(N_i) - \M(N)| \to 0,$$
as $i \to \infty$,
where $\M(\cdot)$ denotes the mass of the underlying integral current.
\end{definition}

In general, if $N_i \toF N$, lower semicontinuity of mass holds: $\liminf_{i \to \infty} \M(N_i) \geq \M(N)$. Thus, the volume-preserving hypothesis in $\VF$-convergence simply assures there is no mass drop in the limit.

We also recall an indispensable result of Sormani and Wenger that establishes the existence of a single ``common space'' into which an $\F$-converging sequence embeds:

\begin{thm}[Theorem 4.2 of \cite{SW}]
\label{thm_embedding}
Suppose $N_i=(X_i, d_i, T_i)$ $\F$-converges to $N=(X,d,T)$. Then there exists a complete metric space $Z$ and distance-preserving maps  $\varphi_i : X_i \to Z$ and $\varphi: X \to Z$ such that $d_F^Z(\varphi_{i\#}(T_i), \varphi_{\#}(T)) \to 0$. Furthermore, by applying the Kuratowski embedding theorem, one may assume without loss of generality that $Z$ is a $w^*$-separable Banach space, i.e.,  $Z=G^*$, the dual space of $G$, where $G$ is a separable Banach space.
\end{thm}
Although the underlying spaces in an $\F$-converging sequence
$N_i = (X_i, d_i, T_i)  \xrightarrow{\F} N=(X,d,T)$
may all be distinct, Sormani \cite{Sor} defines the convergence of points $x_i\in X_i$ to $x \in (\ol X, \ol d)$ as the existence of $Z$ and distance-preserving maps  $\varphi_i$ and $\varphi$ as in Theorem \ref{thm_embedding} for which 
\begin{equation}
\label{eqn_points}
\varphi_i(x_i) \to \varphi(x) \text{ in } Z \text{ as } i \to \infty,
\end{equation}
where $\varphi$ denotes the canonical extension to $\ol X$. This concept will be used in the definition of pointed convergence below. 

Next, observe that the definition of integral current space readily generalizes to the local setting:

\begin{definition} [\cite{JL}]
A \emph{local integral current space} of dimension $m \geq 0$ is a triple $N=(X,d,T)$ in which $(X,d)$ is a metric space and $T$ is a locally integral $m$-current on $(\ol X, \ol d)$ such that $\set(T)=X$. 
\end{definition}

For example, any complete, connected, oriented Riemannian manifold $N^m$ forms a local integral current space with the usual distance function, where $T=[[N]]$ is given as integration as in \eqref{eqn_integration},
with again $\|T\|$ agreeing with Riemannian volume measure.

We gather some basic known results regarding local integral current spaces:

\begin{prop}\label{prop-propertiesIntCurr}
Let $N=(X,d,T)$ be a local integral current space of dimension $m$. \begin{enumerate}[(a)]
\item Given $q \in X$, for almost all $r>0$, $T \llcorner B(q,r)$ is an integral $m$-current on $X$, so in particular, $$N\llcorner B(q,r) := (\set(T \llcorner B(q,r)), d, T \llcorner B(q,r))$$
forms an integral current space. Throughout the paper, the ball in the notation ``$T \llcorner B(q,r)$'' refers to a ball in the completion of $X$.)
\item $\|T\|$ is a Borel measure on $X$ that is finite on bounded sets.
\item $(X,d)$ is countably $\mathcal{H}^m$-rectifiable, that is, there exist at most a countable number of Lipschitz maps $\varphi_i : A_i \subset \mathbb R^m \to X$ such that $\mathcal H^m( X \setminus \bigcup_i \varphi_i(A_i))=0$. If desired, the $A_i$ can be taken to be compact and the $\varphi_i(A_i)$ pairwise disjoint. 
\item There exists a $w^*$-separable Banach space $Y$ such that $(X,d)$ embeds isometrically in $Y$.
\end{enumerate}
\end{prop}

\begin{proof}
Part (a) follows from \cite[Lemma 2.34]{Sor} and \cite[Lemma 13]{JL}, and part (b) is proven in \cite{LW}. Claim (c) follows from (a) and \cite[Remark 2.36]{SW}.

To verify (d), we first recall that $\spt(T)$ is separable. To see this, recall  it is shown in \cite[Lemma 2.9]{AK_cur} that $\spt(T)$ is separable in the case $\|T\|(\ol X) < \infty$. The general case follows by writing $\spt(T)$ as the countable union of separable sets $\spt(T \llcorner B(p,r_i))$ for an appropriate sequence $r_i \to \infty$. By the Kuratowski embedding theorem, we may isometrically embed $\spt(T)$ into $\ell^\infty$, which the dual of $\ell^1$. Since $\spt(T) \supseteq \set(T)=X$, the claim follows.
\end{proof}

 The following definition is taken from \cite{JL}, suitably generalized to possibly incomplete spaces. We also refer the reader to the paper of Takeuchi \cite{Tak} giving a different approach to pointed $\F$-convergence (cf. \cite[Remark 1]{JL}).

 \begin{definition}
 \label{def_pointed_F}
A sequence $N_i=(X_i,d_i,T_i)$ of  local integral current spaces of dimension $m$ converges to a  local integral current space $N=(X,d,T)$ of dimension $m$ in the \emph{pointed Sormani--Wenger intrinsic flat sense} or ``\emph{pointed $\F$-sense}'' (respectively, \emph{pointed volume-preserving intrinsic flat sense} or ``\emph{pointed $\VF$-sense}'') with respect to $p_i \in X_i$ and $q \in \ol{X}$
if
for any $r_0 > 0$, there exists $r \geq r_0$ such that $N \llcorner B(q,r)$ and $N_i \llcorner B(p_i,r)$ are precompact integral current spaces (for all $i$ sufficiently large), and $N_i \llcorner B(p_i,r) \xrightarrow{\F} N \llcorner B(q,r)$ (respectively, $N_i \llcorner B(p_i,r) \xrightarrow{\VF} N \llcorner B(q,r)$), and if $p_i \to q$ as in \eqref{eqn_points}.
 \end{definition}
 
We verify that pointed $\F$-convergence is a reasonable notion, in that limits are unique, up to current-preserving and basepoint-preserving isometry:

 \begin{prop}
 \label{prop_pointed_limit}
Suppose $N_i=(X_i,d_i,T_i)$ is a sequence of  local integral current spaces of dimension $m$ that converges in the pointed $\F$-sense to 
a local integral current space of dimension $m$,  $N=(X,d,T)$, with respect to $p_i \in X$ and $q \in \ol X$. Suppose that $N_i$ also converges in the pointed $\F$-sense to another local integral current space of dimension $m$,  $N'=(X',d',T')$, with respect to $p_i \in X$ and $q' \in \ol X'$. Then there exists a current-preserving isometry $\Phi$ (i.e., $\Phi_{\#}(T) = T'$) from the completion of $(X,d)$ to the completion of $(X',d')$ such that  $\Phi(q)=q'$.
\end{prop}
A similar result appears in \cite[Proposition 3.7]{Tak}.

\begin{proof}
We will first inductively define an increasing sequence of radii $\{r_k^*\}$, diverging to infinity, and a sequence of isometries $\Phi_k : B(q, r_k^*) \to B(q', r_k^*)$ that satisfy $\Phi_k(q) = q'$ and that are current-preserving in the sense that $(\Phi_k)_\# (T \llcorner B(q, r_k^*)) = T' \llcorner B(q', r_k^*)$. Afterwards, we will use a diagonal argument to create a current-preserving isometry $\Phi : X \to X'$.

Let $r_1>0$ be arbitrary. By definition of pointed $\F$-convergence, there exist $r>r_1$ and $r'>r$ such that $N \llcorner B(q,r)$, $N_i \llcorner B(p_i,r)$, $N' \llcorner B(q',r')$, and $N_i \llcorner B(p_i,r')$ are all precompact integral current spaces for all $i$ sufficiently large, and that 
$$N_i \llcorner B(p_i,r) \toF N \llcorner B(q,r)$$ with $p_i \to q$ and 
$$N_i \llcorner B(p_i,r') \toF N' \llcorner B(q',r')$$
with $p_i \to q'$. 

By \cite[Lemma 4.1]{Sor}, we may pass to a subsequence such that for almost all radii $\leq r$ we have convergence in the first equation above, and for almost all radii $\leq r'$, we have convergence in the second equation above. In particular, there exists $r_1^* > r_1$ such that 
$$N_i \llcorner B(p_i,r_1^*) \toF N \llcorner B(q,r_1^*)$$
and
$$N_i \llcorner B(p_i,r_1^*) \toF N' \llcorner B(q',r_1^*)$$
as sequences of precompact integral current spaces, with $p_i \to q$ and $p_i \to q'$, respectively. 
That is, there exist complete metric spaces $Z$ and $Z'$ and distance-preserving maps  $\varphi: B(q,r_1^*) \to Z$, $\varphi_i :B(p_i,r_1^*) \to Z$, $\varphi' : B(q',r_1^*) \to Z'$, and $\varphi_i' :  B(p_i,r_1^*) \to Z'$ and  such that
\[ 
d_F^Z((\varphi_i)_\# (T_i \llcorner B(p_i,r_1^*)), \varphi_{\#} (T \llcorner B(q,r_1^*))) \to 0
\]
and $\varphi_i(p_i) \to \varphi(q)$ in $Z$, and
\[
d_F^{Z'}((\varphi_i')_{\#} (T_i \llcorner B(p_i,r_1^*), (\varphi')_{\#} (T' \llcorner B(q',r_1^*)))) \to 0
\]
and $\varphi'_i(p_i) \to \varphi'(q')$ in $Z'$. 

We apply \cite[Proposition 1.1]{LW} to the sequence of complete metric spaces given by the closed balls $\ol B(p_i,r_1^*)$ in the completions $\ol X_i$, the points $p_i$, the integral currents $T_i \llcorner B(p_i,r_1^*)$, and the distance-preserving maps  $\varphi_i$ and $\varphi_i'$ (which canonically extend to $\ol B(p_i,r_1^*)$). This produces an isometry $$\Phi_1:\{q\} \cup \spt(T \llcorner B(q,r_1^*)) \to \{q'\} \cup \spt(T' \llcorner B(q',r_1^*))$$ that satisfies $\Phi_1(q) = q'$ and is current-preserving: 
$$(\Phi_1)_{\#}(T \llcorner B(q,r_1^*)) = T' \llcorner B(q',r_1^*).$$ We claim that the open ball $B_{\ol X}(q,r_1^*)$ in the completion is a subset of $ \spt(T \llcorner B(q,r_1^*))$. Let $x \in B_{\ol X}(q,r_1^*)$. There exists a sequence $\{x_j\}$ in $X$ that converges in $\ol X$ to $x$. By the triangle inequality, $\ol d(x_j,q) < r_1^*$ for $j$ sufficiently large. Now, since $x_j \in X=\set(T)$, it follows that 
$$x_j \in \set(T \llcorner B(q,r_1^*)) \subseteq \spt(T \llcorner B(q,r_1^*)).$$ Since the support is a closed set, we have $x \in \spt(T \llcorner B(q,r_1^*))$, which shows the claim. Thus, $\Phi_1$ restricts to an isometry $B_{\ol X}(q,r_1^*) \to B_{\ol X'}(q',r_1^*)$, which then extends to an isometry (also denoted $\Phi_1$) $\ol B(q,r_1^*) \to \ol B(q', r_1^*)$ (where again $\ol B$ refers to the closed ball in the completion). Moreover, $\Phi_1$ remains base-point-preserving and current-preserving.

Suppose $r_k^*$ and $\Phi_k$ have been defined for some $k \in \mathbb{N}$. Now take some $r_{k+1} > r_k^* + 1$ and in the same way as above find an $r_{k+1}^*> r_{k+1}$ and a current-preserving isometry $\Phi_{k+1}: B(q, r_{k+1}^*) \to B(q', r_{k+1}^*)$ that maps $q$ to $q'$. Clearly the $r_k^*$ diverge to infinity.

Now that we have created a sequence of base point- and current-preserving isometries $\{\Phi_k\}$ we will perform a diagonal argument to create a base point- and current-preserving isometry $\Phi : \ol X \to \ol X'$. For each $k$, we can restrict the maps $\Phi_{k}$ to maps  from $\ol B(q,r_1^*) \to \ol B(q',r_1^*)$. 
Note that $\ol B(q,r_1^*)$ and $\ol B(q',r_1^*)$ are compact metric spaces by hypothesis, so by the Arzela--Ascoli theorem for functions defined on compact metric spaces taking values in compact metric spaces \cite[Theorem 7.17]{KJ}, 
there exists a subsequence $n^{(1)}$ indexed by $\ell$ such that $\Phi_{n^{(1)}_\ell} : \ol B(q, r_1^*) \to \ol B(q', r_1^*)$ converges uniformly on $\ol B(q, r_1^*)$ as $\ell \to \infty$ to an isometry and clearly the resulting map takes $q$ to $q'$. If for some $j \in \mathbb{N}$, the subsequence $n^{(j)}$ has been defined, we may select a subsequence $n^{(j+1)}$ of $n^{(j)}$ such that the restriction of the 
sequence of functions $\Phi_{n^{(j+1)}_\ell}$ to isometries 
$\ol B(q, r_{j+1}^*)  \to  \ol B(q', r_{j+1}^*)$ converges uniformly on $\ol B(q, r_{j+1}^*)$ as $\ell \to \infty$ to an isometry. 
Finally, we consider the diagonal subsequence $\Phi_{n^{(\ell)}_{\ell}}$. This subsequence converges uniformly on compact sets to some map $\Phi:\ol X \to \ol X'$. It is straightforward to see that $\Phi$ is an isometry and $\Phi(q) = q'$.

We still need to show that $\Phi_\# (T) = T'$. It suffices to show that for every $k \in \mathbb{N}$, 
\[
\Phi_\# (T \llcorner B(q, r_k^*)) = T' \llcorner B(q', r_k^*).
\]
This follows as for every $\ell$ sufficiently large we have
\[ 
(\Phi_{n^{(\ell)}_\ell})_\# (T \llcorner B(q, r_k^*)) = T' \llcorner B(q', r_k^*).
\]
and the isometries $\Phi_{n^{(\ell)}_\ell}$ converge uniformly to $\Phi$. Here, we are using continuity properties of metric currents with locally finite mass: see (2.4) in \cite{LW}.
\end{proof}

\subsection{Dirichlet energy and capacity}
\label{sec_dirichlet}

For integral currents on a Banach space, Portegies gave a definition of the Dirichlet energy of a Lipschitz function based on Ambrosio--Kirchheim's definition of the tangential differential of a Lipschitz function on a rectifiable set in a Banach space. By embedding a metric space in a Banach space, this allowed him to define the Dirichlet energy of a Lipschitz function on an integral current space. To make this precise, we recall from \cite{AK_rect} (with the necessary details included in Appendix B) the notion of the tangential differential $d_x^S f$ of a Lipschitz function $f$ on a $\cH^m$-countably rectifiable set $S$. We then proceed to generalize the Dirichlet energy to a local integral current space $N=(X,d,T)$. Given a Lipschitz function $f:X \to \R$ with bounded support, define (following \cite{Por}):
$$E_N(f) = \int_{X} |d_x ^Xf|^2 d \|T\|(x),$$
where $d_x^X f$ is the tangential differential when $X$ is embedded in some appropriate Banach space. This is well-defined and independent of the embedding. We also note $|d_x ^Xf| \leq \Lip(f)$ and $\|T\|$ is finite on bounded sets, so that $E_N(f)<\infty$. This definition of Dirichlet energy is consistent with the usual definition in the smooth case, again following \cite{Por}:
\begin{prop}
Let $(M,g)$ be a complete, connected, oriented Riemannian manifold, and let $N$ be the associated local integral current space. Then for a Lipschitz function $f$ of $M$ with compact support,
$$E_N(f) = \int_M |\nabla f|^2 dV,$$
where the gradient norm and volume measure are taken with respect to $g$.
\end{prop}

Now, let $N=(X,d,T)$ denote a local integral current space of dimension $m\geq 2$. Let $K \subseteq X$ be a closed, bounded subset. We define
\begin{equation}
\label{def_cap}
\capac_N(K) = \frac{1}{\gamma_m}\inf \{ E_N(f) \; : \; f\in \Lip_B(X),  f \equiv 1 \text{ on a neighborhood of } K\},
\end{equation}
where $\gamma_m = (m-2)\omega_{m-1}$ for $m \geq 3$ and $\gamma_2 = 2\pi$. It is clear that $\capac_N(K) \in [0, \infty)$ and that $K_1 \subseteq K_2 $ implies $\capac_N(K_1) \leq \capac_N(K_2)$. 
In Euclidean space, this agrees with the usual definition of capacity stated in the introduction. (The latter required merely that  $f\equiv 1$ on $K$ itself, but in $\R^m$ the distinction is immaterial see, e.g., 
 \cite[Section 4.7]{EG}.) For example, a ball of radius $r$ in $\R^m$ has capacity $r^{m-2}$ for $m\geq 3$. Note that capacity is only interesting (i.e., not identically zero) in the case in which $X$ is unbounded, and even then it is possible that $\capac_N \equiv 0$ depending on the behavior of $X$ and $\|T\|$ ``at infinity.''

\begin{remark}
Given a competitor $f$ for the capacity, replacing $f$ with its truncation between values of 0 and 1 produces another competitor whose Dirichlet energy has not increased. Thus, we may restrict to functions $f$ in \eqref{def_cap} satisfying $0 \leq f \leq 1$.
\end{remark}

\begin{remark}
One can similarly define the $p$-capacity, for $p \geq 1$, by replacing $|d_x^X f|^2$ in the definition of $E_N$ with $|d_x^X f|^p$.
\end{remark}

\section{Semicontinuity of capacity}
\label{sec_main}

The main results of this paper are Theorems \ref{thm_balls} and \ref{thm_sublevel} below, where Theorem \ref{thm_balls} is a restatement of Theorem \ref{thm_balls_intro} from the introduction. In both cases we assume pointed $\VF$-convergence of local integral current spaces as in Definition  \ref{def_pointed_F} and establish the upper semicontinuity of the capacity of sets in the spaces. In the first case, the sets are balls centered around converging points; in the second, the sets are defined as Lipschitz sublevel sets. Both theorems will ultimately be consequences of the main technical result, Theorem \ref{thm_extrinsic}.

\subsection{Corresponding regions}
Before presenting the theorems, we will recall the construction in \cite{JL} of ``corresponding regions.'' Let $N_i = (X_i, d_i, T_i)$ and $N = (X, d, T)$ be a sequence of local integral current spaces and a local integral current space, respectively, all of dimension $m \geq 2$, and let $p_i \in X_i$ and $p \in \overline X$. We assume the following property holds: for any $r_0 > 0$, there exists $r \geq r_0$, a $w^*$-separable Banach space $Y$, and distance-preserving maps $\varphi_i: \set(T_i \llcorner B(p_i,r)) \to Y$ and $\varphi: \set(T \llcorner B(p,r)) \to Y$ such that $\varphi_i(p_i) \to \varphi(p)$ in $Y$ as $i \to \infty$. 
Although we do not yet assume pointed $\F$-convergence, we see that this property holds whenever $N_i  \to N$  in the pointed $\mathcal{F}$-sense, by definition and Theorem \ref{thm_embedding}. Suppose $K \subsetneq X$ is 
nonempty and compact. The corresponding regions, to be defined, will be subsets $K_i$ of $X_i$.

Fix a function $u:X \to \R$ with
\begin{align}\label{defining_function}
\begin{cases}
& \{u \leq 0\} = K, \\
& \Lip(u)=1, \text{ and}\\
& \text{for all open } O \supseteq K, \text{ there exists } \beta \in (0,1] \text{ such that } u(x) \geq \beta d(x,K) \text{ for all } x \in X \setminus O.
\end{cases}
\end{align}

Such a function $u$ will be called a \emph{defining function} for $K$. For example, $u(x)= d(x, K)\geq 0$ for $x \in X$ is such a function with $\beta=1$, but the definition allows for, for example, a signed distance function to $\partial K$ if it can be well-defined and if it is a 1-Lipschitz function. As another example, if $K=\ol B(p,r_1)$, then $u(x) = d(p,x) -r_1$ serves as a defining function for $K$, provided $\ol B(p,r_2)$ is compact for some $r_2>r_1>0$. This fact will be verified in the proof of Theorem \ref{thm_balls}.

We fix $r_0$ so that $B(p,r_0) \supset K$ in $X$, and obtain $r\geq r_0$ and appropriate $Y$, $\varphi_i$, and $\varphi$ as above. 
Let  $U: Y \to \R$  be the standard 1-Lipschitz extension\footnote{When we refer to the ``standard Lipschitz extension'' of a Lipschitz function $f:A \subset Y \to \R$, we mean the extension $\tilde f:Y \to \R$ given by
$$\tilde f(y) := \inf_{a\in A} \left(f(a) + \Lip(f) d_Y(a,y)\right),$$
which has $\Lip(\tilde f)=\Lip(f)$.
} of $u\circ \varphi^{-1}$, where the latter is defined on the image of $\varphi$. We then define $u_i = U \circ \varphi_i$, a 1-Lipschitz function on $\set(T_i \llcorner B(p_i,r))$ for each $i$. 
Given a sequence of nonnegative real numbers $\{\alpha_i\}$ converging to 0, the sets
\begin{equation}
\label{eqn_K_i}
K_i = u_i^{-1}(-\infty,\alpha_i] \subseteq X_i
\end{equation}
will be called a sequence of \emph{corresponding regions}; they depend on the choices of $u$, $r_0$, $r$, the space $Y$, the maps $\varphi_i$ and $\varphi$, and the sequence $\alpha_i$. They were essentially defined in \cite{JL}, where it was proved that (roughly --- see the proof of Proposition \ref{prop_K_i} below) if $N_i \to N$ in the pointed $\F$ sense, then $N_i \llcorner K_i$ subsequentially $\F$-converges to $N \llcorner K$. (Here we are generalizing the definition in \cite{JL} slightly by allowing the $\alpha_i$ to depend on $i$ as well as working in \emph{local} integral current spaces $N_i$, though we are restricting the choice of defining function $u$.) 

\medskip
Let us try to provide some intuition for the sets $K_i$. We claim that 
$$\varphi_i(\set(T_i \llcorner B(p_i,r)))
\cap  \{ y \in Y \, |\, 
d_Y(y, \varphi(K)) \leq \alpha_i\} 
\subseteq \varphi_i(K_i).$$ 
Indeed let $x \in \set(T_i \llcorner B(p_i,r))$, define $y := \varphi_i(x)$ and assume that $d_Y(y, \varphi(K)) \leq \alpha_i$. Then, since $u \leq 0$ on $K$,
\[
\begin{split}
u_i(x) &= \inf_{a \in \varphi(\set(T \llcorner B(p,r)))}[u \circ \varphi^{-1}(a) + d_Y(a, y) ] \leq 
\inf_{a \in \varphi(K)}[u\circ \varphi^{-1}(a) + d_Y(a, y) ]\\
&\leq
\inf_{a \in \varphi(K)} d_Y(a, y)
=
d_Y(y, \varphi(K) )
\leq \alpha_i,
\end{split}
\]
so $x \in K_i$.
In case the defining function is chosen to be $u := d(\cdot, K)$, then the sets $\varphi_i(K_i)$ are \emph{exactly} the intersections of $\varphi_i(\set(T_i \llcorner B(p_i,r)))$ with the closed $\alpha_i$-neighborhood of $\varphi(K)$.

\medskip

We now establish, in Proposition \ref{prop_K_i}, conditions to assure that the sets $K_i$ are not ``too small'' when we have pointed $\F$-convergence --- a priori they could even be empty. Note that we include  (b) below to accommodate applications that may involve a signed distance function.

\begin{prop}
\label{prop_K_i}
Let $N_i = (X_i, d_i, T_i) \to N = (X, d, T)$  in the pointed $\mathcal{F}$-sense as local integral current spaces of dimension $m \geq 2$
with respect to $p_i \in X_i$ and $p \in \overline X$ as in Definition  \ref{def_pointed_F}. Suppose $K \subsetneq X$ is 
nonempty and compact, fix some defining function $u$ as in \eqref{defining_function}, some
 $r \geq r_0$ so that  $N_i \llcorner B(p_i,r) \toF N \llcorner B(p, r)$ holds, and appropriate distance-preserving maps $\varphi$ and $\varphi_i$ as above.  Then:
\begin{enumerate}
    \item [(a)] There exists a positive sequence $\{\alpha_i\}$ converging to zero, such that when we define corresponding regions, $K_i$, as in \eqref{eqn_K_i}, there exists a subsequence $N_{i_k}$ such that the corresponding regions $K_{i_k}$ for this subsequence satisfy:
$$\liminf_{k \to \infty} \|T_{i_k}\|(K_{i_k}) \geq \|T\|(K).$$
    \item [(b)] Given the choice $\alpha_i=0$ and the corresponding regions $K_i$, there exists a subsequence $N_{i_k}$ such that the corresponding regions $K_{i_k}$ for this subsequence satisfy:
$$\liminf_{k \to \infty} \|T_{i_k}\|(K_{i_k}) \geq \|T\|(\{u < 0\}).$$
\end{enumerate}
\end{prop}

\begin{proof}
(a). 
For $\delta \in \R$, let $K^\delta = \{u \leq \delta\}$. By \cite[Lemma 27]{JL}, we may pass to a subsequence (keeping the same indexing) such that the restriction of $T_i$ to $K_i^\delta= \{u_i \leq \delta\}$ $\F$-converges to $T \llcorner K^\delta\neq 0$ for almost every $\delta \in \R$. By lower semicontinuity of mass under $\F$-convergence, we have for $\delta \geq 0$:
\begin{equation}
\label{eqn_K_i_delta}
\liminf_{i \to \infty} \|T_i\|(K_i^\delta) \geq \|T\|(K^\delta) \geq \|T\|(K).
\end{equation}
Select $\delta_1 \in (2^{-1}, 2^0)$ for which \eqref{eqn_K_i_delta} holds with $\delta=\delta_1$. Then there exists $i_1 \geq 1$ such that
$$\|T_{i_1}\|(K_{i_1}^{\delta_1}) \geq \|T\|(K) \left(1-2^{-1}\right).$$
Using \eqref{eqn_K_i_delta} repeatedly, we iteratively select $\delta_k \in (2^{-k}, 2^{-k+1})$ and $i_k > i_{k-1}$ such that
$$\|T_{i_k}\|(K_{i_k}^{\delta_k}) \geq \|T\|(K) \left(1-2^{-k}\right).$$
Since $K_{i_k}^{\delta_k}= u_{i_k}^{-1}(-\infty, \delta_{k}]$, the claim follows with $\alpha_{i_k}=\delta_k$.\\
\medskip 
\indent (b) Given $\epsilon>0$, since $\|T\|$ is a Borel measure, we may select $\delta_0< 0$ such that
\begin{equation}
\label{eqn_delta-eps}    
\|T\|(\{u < \delta_0\}) \geq \|T\|(\{u < 0\})-\epsilon.
\end{equation}
As in the proof of (a), by \cite[Lemma 27]{JL}, we may pass to a subsequence (keeping the same indexing) such that the restriction of $T_i$ to $K_i^\delta= \{u_i \leq \delta\}$ $\F$-converges to $T \llcorner K^\delta\neq 0$ for almost all $\delta \in (\delta_0,0)$. By lower semicontinuity of mass, we have
$$\liminf_{i \to \infty} \|T_i\|(K_i^\delta) \geq \|T\|(K^\delta).
$$
However, for $K_i$ defined using the sequence $\alpha_i=0$, $K_i \supseteq K_i^\delta$ and $K^\delta \supseteq \{u < \delta\}\supset \{u < \delta_0\}$, so by \eqref{eqn_delta-eps} 
$$\liminf_{i \to \infty} \|T_i\|(K_i) \geq \|T\|(\{u < 0\})-\epsilon.$$

From this and \eqref{eqn_delta-eps}, the claim follows by letting $\epsilon \searrow 0$ and applying a diagonal argument.
\end{proof}

\begin{ex}\label{rmrk-corresponding}
In the case in which $N_i \to N$ in the pointed $\F$-sense ``with a common space'' (defined below) with respect to $p_i \in X_i$ and $p\in \ol X$, we can give an alternative definition of corresponding regions $K_i$ that do not depend on the choice of ball radius $r_0$ or $r$. That is, there exists a ``uniform'' method to define corresponding regions. 
By ``with a common space'' we mean when there exist a $w^*$-separable Banach space $Y$ and distance-preserving maps  $\varphi_i:X_i \to Y$ and $\varphi:X \to Y$ such that  $\varphi_{i\#}(T_i) \to \varphi_{\#}(T)$ in the local flat topology (as in Definition \ref{def_local_flat}), and for which $\varphi_i(p_i) \to \varphi(p)$. 
For $K \subsetneq X$ 
nonempty and compact and $u:X \to \R$ a defining function for $K$, we still define 
$$u_i = U \circ \varphi_i: X_i \to \R$$ 
(but note that now $u_i$ is defined on the whole $X_i$ and not only on $\set(T_i \llcorner B(p_i,r))$, and let $K_i$ be as in \eqref{eqn_K_i}. We will use this version of corresponding regions in Theorem \ref{thm_sublevel}.
\end{ex}

\subsection{Main results and proofs}

The following two theorems are our main results. The first is simply a restatement of Theorem \ref{thm_balls_intro} from the introduction.

\begin{thm}
\label{thm_balls}
Let $N_i = (X_i, d_i, T_i)$ and $N= (X, d, T)$ be local integral current spaces of dimension $m \geq 2$, such that $N_i \to N$ in the pointed $\mathcal{VF}$ sense with respect to $p_i \in X_i$ and $p \in X$. Suppose for some $r'>0$ that the closed ball $\ol B(p,r')$ in $X$ is compact. Then for all $0<r<r'$,
\begin{equation}
\label{eqn_limsup_balls}
\limsup_{i \to \infty} \capac_{N_i} (\ol B(p_i,r)) \leq \capac_{N}(\ol B(p,r)).
\end{equation}
\end{thm}

\begin{remark}
\label{rmk_length_space}
If $(X,d)$ is assumed to be a length space, then \eqref{eqn_limsup_balls} in Theorem \ref{thm_balls} holds for $r=r'$ as well. The only place where the hypothesis that $\ol B(p,r)$ is contained in a strictly larger compact ball will be in verifying that $u(x)= d(x,p)-r$ is a defining function, i.e., satisfies \eqref{defining_function}. But it is elementary to verify that in length spaces, $u(x)= d(x,\ol B(p, r))$ for $x \in X \setminus \ol B(p, r)$.
\end{remark}

In the second main theorem, we rely on Lang--Wenger's compactness theorem \cite{LW} to produce a ``common space'' as in Example \ref{rmrk-corresponding}. So, given a fixed set $K$ in the limit space $X$, we consider a sequence of corresponding regions as in this example.
We note that pointed $\VF$ convergence with a common space implies pointed convergence of the underlying metric measure spaces.

\begin{thm}
\label{thm_sublevel}
Let $N_i = (X_i, d_i, T_i)$ and $N = (X, d, T)$ be  local integral current spaces of dimension $m \geq 2$, such that $N_i \to N$ in the pointed $\mathcal{VF}$ sense with respect to $p_i \in X_i$ and $p \in \overline X$. Assume that for all $r>0$,
\begin{equation}
\label{eqn_bdry_mass_bound}    \sup_{i \in \mathbb N} \|\partial T_i\| (B(p_i,r)) < \infty.
\end{equation}
Let $K \subsetneq X$ be nonempty and compact and $u:X \to \R$ a defining function for $K$. Then, passing to a subsequence of $N_i$ that we do not relabel, 
there exist a $w^*$-separable Banach space $Y$ and distance-preserving maps  $\varphi_i:X_i \to Y$ and $\varphi:X \to Y$ such that  $\varphi_{i\#}(T_i) \to \varphi_{\#}(T)$ in the local flat topology and  $\varphi_i(p_i) \to \varphi(p)$. Defining $u_i = U \circ \varphi_i$, where $U$ is the standard Lipschitz extension of $u \circ \varphi^{-1}$ to $Y$, and letting $K_i \subseteq X_i$ be given as in \eqref{eqn_K_i}, i.e.
\begin{equation*}
K_i = u_i^{-1}(-\infty,\alpha_i],
\end{equation*}
we have
\begin{equation}
\label{eqn_limsup_sublevel}
\limsup_{i \to \infty}\capac_{N_i} (K_i) \leq \capac_{N}(K).
\end{equation}
\end{thm}

\medskip
Note the hypothesis \eqref{eqn_bdry_mass_bound} is trivially satisfied if $\partial T_i=0$ for each $i$, or, more generally, if the boundary masses are uniformly bounded.

\medskip

We will first prove an extrinsic version of these theorems (namely Theorem \ref{thm_extrinsic}), i.e. for a sequence of locally integral currents all on a fixed Banach space. We first list the assumptions needed for this version. 
\begin{assum}\label{assum_extrinsic}
    Let $Y$ be a $w^*$-separable Banach space, and let $T\neq 0$ and $T_i$, $i \in \mathbb N$, be locally integral $m$-dimensional currents on $Y$, $m \geq 2$, such that  
$T_i \to T$ weakly (see Definition \ref{def_local_flat}). Assume that for some (and thus every) $z \in Y$ and every $r > 0$, there exists a bounded open set $V \supseteq B(z,r)$ such that $\|T_i\|(V) \to \|T\|(V)$. 
Let $S=\set(T)$ and $S_i = \set(T_i)$, $i \in \mathbb N$, and 
note that $N_\infty=(S,d_Y,T)$ 
and $N_i=(S_i, d_Y, T_i)$ are
$m$-dimensional local integral current spaces.
Let $K \subsetneq S$ be a nonempty compact set.
Let $u:S \to \R$ be a defining function for $K$ as in \eqref{defining_function},
and let  $U: Y \to \R$  be the standard 1-Lipschitz extension of $u$. 
Let $\{\alpha_i\} $ be a sequence of nonnegative real numbers converging to 0, and define 
\begin{equation}
\label{eqn_K_i_U}
K_i = U^{-1}(-\infty,\alpha_i] \cap S_i.
\end{equation}
Let $f \in \LipB(S)$, with $0 \leq f \leq 1$ and $f \equiv 1$ on a neighborhood $O$ of $K$. Let $\beta \in (0,1]$ be a constant in \eqref{defining_function} for the neighborhood $O$, and let $\Lambda=\beta^{-1}\Lip(f)$. Fix some $z_0 \in Y$. 
Fix $r_0>0$ so that $\spt(f) \subseteq B(z_0,r_0)$ and $r_0 >3\diam(\hat K) + d_Y(K, S \setminus K)$, where $\hat K= K \cup \{z_0\}$.
\end{assum}

Some remarks:
\begin{enumerate}
\item
Generally $d_Y(K, S \setminus K)$ is zero, but it may be positive, e.g. if $K$ is a connected component of $S$. 
\item The $K_i$ defined in \eqref{eqn_K_i_U} are the same as the version of corresponding regions in Example \ref{rmrk-corresponding}, where here the embeddings are simply the inclusion maps.  
\end{enumerate}

We now state Theorem \ref{thm_extrinsic}, whose proof follows many of the ideas in the proof of  \cite[Theorem 6.2]{Por}.

\begin{thm}\label{thm_extrinsic}
Suppose Assumptions \ref{assum_extrinsic} hold. First, each $K_i \subseteq S_i$ is a closed and bounded subset of $S_i$. Second, there exists (for $i$ sufficiently large)  a sequence of functions
\[  
f_i \in \LipB(S_i), \quad 0 \leq f_i \leq 1 , \quad f_i \equiv 1 \textrm{ on a neighborhood of }K_i, 
\]
\[
\Lip(f_i) \leq 1+ 3\Lambda, \quad
\spt(f_i) \subseteq B(z_0,r_0+3),
\]
such that
\begin{equation}
\label{USC_energy}
\limsup_{i \to \infty} E_{N_i}(f_i) \leq E_{N_\infty}(f).
\end{equation}
Third:
\begin{equation}
\label{USC_extrinsic}
\limsup_{i \to \infty} \capac_{N_i}(K_i) \leq \capac_{N_\infty}(K).
\end{equation}
\end{thm}

Towards the proof of Theorem \ref{thm_extrinsic} we first prove some lemmas and propositions.

\begin{lemma}\label{lem-K_ibdd}
Under Assumptions \ref{assum_extrinsic},  $K_i = U^{-1}(-\infty,\alpha_i] \cap S_i$  (as in \eqref{eqn_K_i_U}) is a closed and  bounded subset of $S_i$.
\end{lemma}

\begin{proof}
It is immediate that $K_i$ is closed in $S_i$.

Let $O_1$ be the open $1$-neighborhood of $K$ in $S$. By the definition of defining function, there exists a $\beta_1 \in (0,1]$ such that for all $s \notin O_1$, $d(s, K) \leq \beta_1^{-1} u(s)$.

Choose a constant $\alpha > \sup_i\alpha_i$.
Let $p \in K_i$, so $U(p) \leq \alpha_i$. Then by definition of $U(p)$, there exists $s \in S$ such that $u(s) + d_Y(s,p) \leq \alpha$.  Since $K=\{u \leq 0\}$ is compact, $u$ is bounded below by $-L$ for some $L>0$. Then $d_Y(s,p) \leq \alpha+L$. Note that we also have $u(s) \leq \alpha$.

Note that by the choice of $\beta_1$, we have that for all $s \notin O_1$, $d(s, K) \leq \beta_1^{-1} u(s) \leq \frac{\alpha}{\beta_1}$. Moreover, $d_Y(s,K) < 1$ for all $s \in O_1$.
Let $k \in K$ minimize the distance from $s$ to $K$. Then we obtain:
\begin{align*}
d_Y(p,z_0) &\leq d_Y(p,s) + d_Y(s,k) + d_Y(k,z_0)\\ 
&\leq \alpha + L + \max\left(1, \frac{\alpha}{\beta_1}\right) + \diam(\hat K),
\end{align*}
a finite number independent of $p$. Thus $K_i$ is bounded.
\end{proof}

Proposition \ref{prop-Energy} below is a weaker version of Theorem \ref{thm_extrinsic} that will ultimately be used to prove the latter. 

\begin{prop}\label{prop-Energy}
Suppose Assumptions \ref{assum_extrinsic} hold, 
and let $\epsilon > 0$.  
There exists (for $i$ sufficiently large)  a sequence of functions  $f_i^\epsilon \in \LipB(S_i)$, $0 \leq f_i^\epsilon \leq 1$, with $f_i^\epsilon \equiv 1$ on a neighborhood of $K_i$ (for $i$ sufficiently large, depending on $\epsilon$) such that 
\[ 
\Lip(f_i^\epsilon) \leq 1+ 3\Lambda, \quad \spt(f_i^\epsilon) \subseteq B(z_0,r_0+3)
\]
and 
\begin{equation}
\label{USC_energy_with_eps}
\limsup_{i \to \infty} E_{N_i}(f_i^\epsilon) \leq E_{N_\infty}(f) + \epsilon.
\end{equation}
\end{prop}

 In order to prove Proposition \ref{prop-Energy}, we will first state and prove Lemmas \ref{lem-extensionf}--\ref{lemma_extension}.
For clarity of notation, we drop the superindex $\epsilon$ in what follows.

\begin{lemma}\label{lem-extensionf}
Under Assumptions \ref{assum_extrinsic}, 
for $\gamma>0$ sufficiently small, the extension of $f$ from $S$ to $S \cup U^{-1}(-\infty,\gamma]$, defined as 1 in $U^{-1}(-\infty,\gamma] \setminus S$, is a Lipschitz function, bounded between 0 and 1, with Lipschitz constant $\leq 2\Lambda$. 
(We will also call the extension $f$, so now $\Lip(f) \leq 2\Lambda$.) 
\end{lemma}

 Please refer to Figure \ref{fig1}  for an illustration of some aspects of this setup.

 \begin{figure}
    \centering
    \includegraphics[scale=0.7]{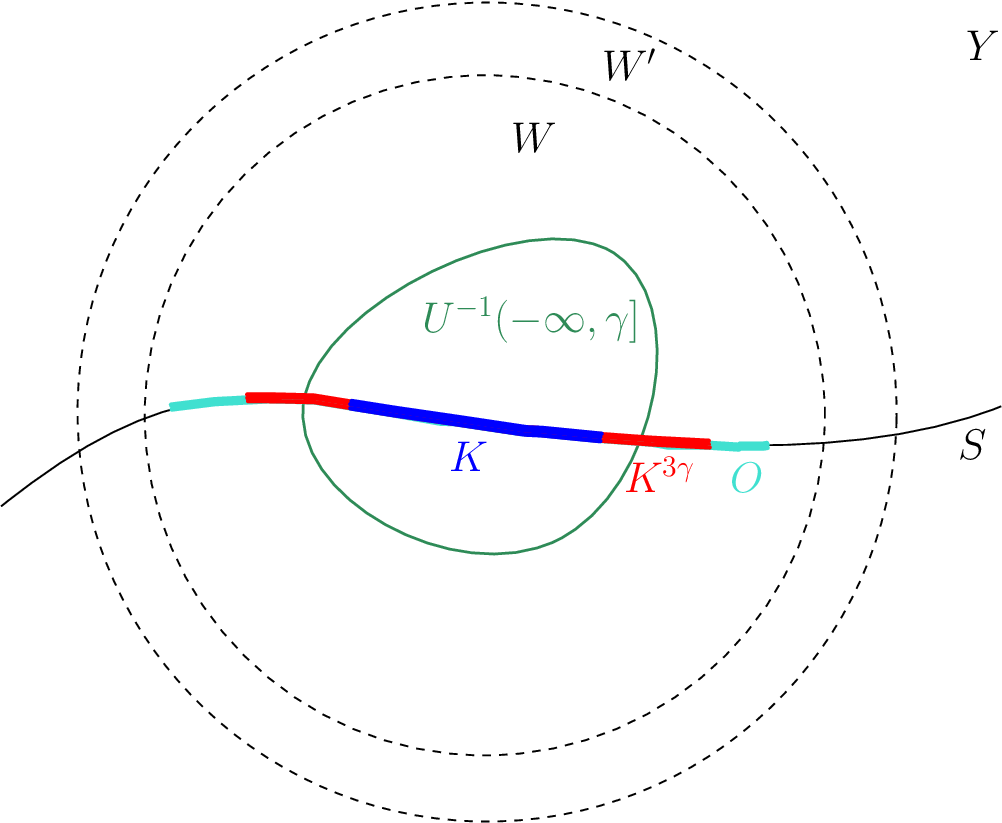}
    \caption{The sets $K \subset S =\set(T) \subset Y$ are defined in Theorem \ref{thm_extrinsic}. The function $f: S \to \R$ from Proposition \ref{prop-Energy} is identically $1$ on $O \subset S$ which contains $K$. In Lemma \ref{lem-extensionf}, $\gamma>0$ is chosen so that $K \subseteq K^{3\gamma} \subseteq O$, and $f$ is extended as a constant $1$ to  $U^{-1}(-\infty, \gamma] \subset Y$. In Lemma \ref{lem:Kdeltas}, we introduce open balls 
    $W \subset W'\subset Y$.}
    \label{fig1}
\end{figure}

\begin{proof}
Recall from Assumptions \ref{assum_extrinsic} that $O \subset S$ denotes a fixed neighborhood of $K$ on which $f \equiv 1$, and $\beta \in (0,1]$ is fixed such that $u(x) \geq \beta d(x,K)$ for $x \in S \setminus O$. Since $K$ is compact and $S \setminus O$ is closed, we then see $u(x)$ is bounded below by a positive constant on $S \setminus O$. In particular, we can choose $\gamma>0$ small enough to ensure $K^{3\gamma}=u^{-1}(-\infty,3\gamma] \subseteq O$.

The extension $f:S \cup U^{-1}(-\infty,\gamma] \to \R $ obviously satisfies $0 \leq f \leq 1$. We now prove it is Lipschitz. Let $p, q \in S \cup U^{-1}(-\infty,\gamma]$. If $p,q \in S$, then $|f(p)-f(q)|$ is bounded above by $\Lip(f|_S) d_Y(p,q)\leq\Lambda d_Y(p,q)$. If $p,q \in (U^{-1}(-\infty,\gamma]\setminus S)\cup O$, then $f(p) - f(q) = 1 - 1 = 0$. 
Therefore, the only case left to consider is that in which $p \in (U^{-1}(-\infty,\gamma] \setminus S)\cup O$ and $q \in S$. In fact, since $O \subset S$ we may assume that $p \in U^{-1}(-\infty, \gamma] \setminus S$ and $q \in S \setminus O$.

By the definition of defining function, since $q \in S \setminus O$, we have $u(q) \geq \beta d(q,K)$. Since $K$ is compact, there exists $k \in K$ such that $d(q,k) = d(q,K)$. Thus,
$$d_Y(k,q) \leq \beta^{-1} u(q).$$

Then since $f(p)=1=f(k)$ and $f|_{S}$ is Lipschitz, we obtain:
\begin{align*}
|f(p) - f(q)| &= |f(k) - f(q)|\\
&\leq\Lip(f|_S) d_Y(k, q)\\
&\leq \Lip(f|_S) \beta^{-1} u(q) = \Lambda u(q).
\end{align*}

We now bound $u(q)=  U(q)$ in terms of $d_Y(p,q)$. Since $U$ is 1-Lipschitz, we have
$$U(q) - U(p) \leq |U(q)-U(p)| \leq d_Y(p,q),$$
so that
$$u(q) \leq d_Y(p,q) + \gamma.$$
We now show that $\gamma< d_Y(p,q)$.  From the definition of $U(p)\leq \gamma$, there exists $s \in S$ such that
\begin{equation}
\label{eqn_2c}
u(s) + d_Y(s,p) < 2\gamma.
\end{equation}
Since $q \in S \setminus O$, we have $u(q)>3\gamma$, and thus 
$$3\gamma - u(s) < u(q) - u(s) \leq |u(q)-u(s)| \leq d_Y(q,s) \leq d_Y(p,q) + d_Y(s,p) < d_Y(p,q) + 2\gamma - u(s),$$
having used that $u$ is 1-Lipschitz and \eqref{eqn_2c} to bound $d_Y(s,p)$. From this, it follows that $\gamma< d_Y(p,q)$.  Thus, 
$$|f(p) - f(q)| \leq \Lambda(d_Y(p,q) + \gamma) \leq  2\Lambda d_Y(p,q).$$
This completes the proof.
\end{proof}

 The following technical lemma will allow us to construct many of the objects used in the proof of Proposition \ref{prop-Energy}. See again Figure \ref{fig1}.

\begin{lemma}\label{lem:Kdeltas} 
Let $\epsilon_1>0$ be given. Under Assumptions \ref{assum_extrinsic}, there exist open balls $W, W' \subset Y$ about $z_0$ of radii in $(r_0, r_0+1)$ and  $(r_0+2, r_0+3)$, respectively, and $\delta>0$ sufficiently small, so that, letting  $S'=S\cap W'$,
\begin{equation}
\label{eqn_bdry_W'}
T \llcorner S' \in \Im(Y),   \quad \|T\|(\partial W)=0, \quad \;\; \|T\|(\partial W')=0,
\end{equation}
$U^{-1}(-\infty, \delta] \subseteq W$, and 
$K^\delta=u^{-1}(-\infty, \delta]$ satisfies
\begin{equation}\label{eq-epsilon2}
\|T\|(K^\delta \setminus K) < \epsilon_2, \ \quad T \llcorner K^\delta \in \Im(Y), \quad \text{ and } \;\; T \llcorner (S' \setminus K^\delta) \in \Im(Y),
\end{equation}
where $\epsilon_2 = (1+3\Lambda)^{-2} \epsilon_1$.
\end{lemma}

\begin{proof}
    We note that
since $\|T\|$ is a Borel measure that is finite on bounded open sets, it follows that $\|T\|$ is zero on almost all metric spheres about a given point.  Choose an open ball $W \subset Y$ about $z_0$ of radius in $(r_0,r_0+1)$, with the radius chosen so that $\|T\|(\partial W)=0$.

We claim that for $\delta>0$ sufficiently small, $U^{-1}(-\infty, \delta] \subseteq W$.
The proof is similar to the proof of each $K_i$ being bounded: Let $p \in U^{-1}(-\infty,\delta]$, i.e., $U(p) \leq \delta$. Then given $\eta > 0$, there exists $s \in S$ such that $u(s)+d_Y(s,p)\leq  \delta+ \eta$. In particular, $u(s) \leq \delta+\eta.$ 

We consider two cases according to the sign of $u(s)$. If $u(s) \geq 0$,  we obtain that $d_Y(s,p)\leq  \delta+ \eta$.
Taking $k \in K$ to minimize the distance to $s$, we obtain by the triangle inequality
\begin{align*}
d_Y(z_0,p) &\leq d_Y(z_0,k) + d_Y(k,s) + d_Y(s,p)\\
&\leq \diam(\hat K) + d_Y(k,s) + \delta + \eta.
\end{align*} 
To estimate the middle term, let $O_1$ be the open 1-neighborhood of $K$ in $S$, and let $\beta_1 \in (0,1]$ be a constant guaranteed by the definition of defining function, \eqref{defining_function}, for the open set $O_1$. If $s \in O_1$, then $d_Y(k,s) < 1$; if $s \not \in O_1$, then $d_Y(k,s) = d(K,s) \leq \beta_1^{-1} u(s) \leq \frac{\delta+\eta}{\beta_1}.$
In particular, choosing $\delta$ and $\eta$ sufficiently small, we have  $$d_Y(z_0,p) \leq \diam(\hat K) + 1 < r_0,$$
so $p \in W$ in this case.

In the other case, if $u(s) < 0$, then $s \in K$ and so $d_Y(z_0,s) \leq \diam(\hat K)$. We also have 
$d_Y(s,p) \leq \delta + \eta + |u(s)|$. Let $q \in K$ and $q' \in S\setminus K$ achieve the minimum distance between $K$ and $S\setminus K$, within $\eta$:
$$d_Y(q,q') \leq d_Y(K, S \setminus K) + \eta.$$
Since $u(q') > 0$, 
$$|u(s)| \leq |u(q') - u(s)| \leq d_Y(q', s) \leq d_Y(q',q) + d_Y(q,s) \leq d_Y(K, S \setminus K) + \eta + \diam(K).$$
Thus,
$$d_Y(s,p) \leq \delta + \eta + |u(s)| \leq \diam(K) + d_Y(K,S \setminus K) + \delta+2\eta.$$
Again, by the triangle inequality and choosing $\delta$ and $\eta$ sufficiently small, we obtain 
$$d_Y(z_0,p) \leq d_Y(z_0,s) + d_Y(s,p) \leq \diam(\hat K)  + \diam(K) + d_Y(K,S \setminus K)  + \delta+2\eta < r_0,$$ 
by Assumptions \ref{assum_extrinsic},
so $p \in W$. It follows that $U^{-1}(-\infty, \delta] \subseteq W$. 

\medskip

We now choose $\delta >0$ as small as needed so that $U^{-1}(-\infty, \delta] \subset W$ and that $K^\delta$
satisfies
\begin{equation*}
\|T\|(K^\delta \setminus K) < \epsilon_2 \quad \text{ and } \quad T \llcorner K^\delta \in \Im(Y),
\end{equation*}
where $\epsilon_2 = (1+3\Lambda)^{-2} \epsilon_1$
(cf. \cite[Lemma 24]{JL}, using Sormani's argument in \cite[Lemma 2.34]{Sor}).  

\medskip

Now fix another open ball $W' \subset Y$, centered about $z_0$, of radius in $(r_0+2,r_0+3)$, so that in particular $\ol{W} \subset W'$. Let $S'=S\cap W'$; adjusting the radius if necessary, 
we may ensure $T \llcorner S' \in \Im(Y)$ and 
\begin{equation*}
\|T\|(\partial W')=0.
\end{equation*}
So $T \llcorner (S' \setminus K^\delta) \in \Im(Y)$ as well.
\end{proof}

Since the canonical set of an integral current is countably rectifiable and by the properties of the metric differential of Lipschitz functions, Portegies obtained good bi-Lipschitz charts (Lemmas 3.2 and 6.1 of \cite{Por}) for integral currents and Lipschitz functions defined on their support. We now apply these results to $T \llcorner (S' \setminus K^\delta)$ and the restriction of $f$ to $S' \setminus K^\delta$.
 Refer to Figures \ref{fig1} 
  and \ref{fig2} for an illustration of some aspects of this setup.

  \begin{figure}
    \centering
    \includegraphics[scale=0.7]{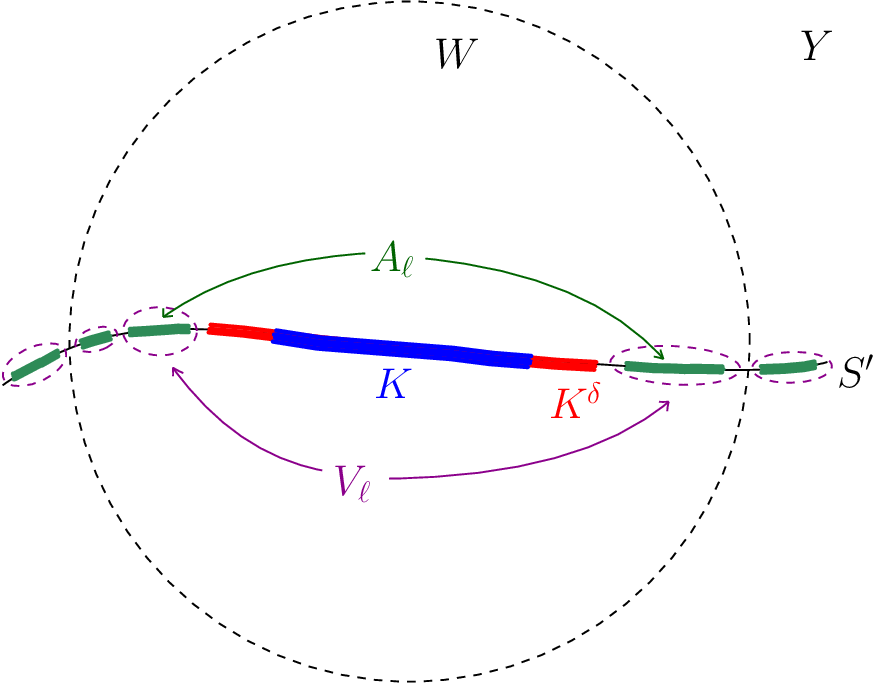}
    \caption{The finitely many sets $A_\ell$ are subsets of $S'\setminus K^\delta$ and are defined in Lemma \ref{lem:setsAandV}. They cover all of $S'\setminus K^\delta$ except for $\|T\|$-measure $< \epsilon_4$. The $V_\ell \subset Y$ are disjoint neighborhoods of the $A_\ell$, each of which lies in $W$ or $Y \setminus \ol W$. The larger ball $W'$ is not pictured, but $S' = S \cap W'$.
    } 
    \label{fig2}
\end{figure}

\medskip

\begin{lemma}\label{lem:setsAandV}
Let $\epsilon_1>0$ be given. Suppose Assumptions \ref{assum_extrinsic} hold. 
Let  $W, W' \subset Y$ be balls about $z_0$ of radius in $(r_0, r_0+1)$ and $(r_0+2, r_0+3)$, respectively,  and $\delta>0$ be given as in Lemma \ref{lem:Kdeltas}. Then there exists a finite collection of compact sets $A_{\ell} \subset S' \setminus K^\delta$ for $\ell=1,\ldots, N$ such that:
\begin{itemize}
\item each $A_\ell$ is the bi-Lipschitz image of a compact subset of $\R^m$,
\item each $A_\ell$ is either a subset of $W$ or of $Y \setminus \ol W$,
\item the $A_{\ell}$ are pairwise disjoint with minimum pairwise distance $a>0$,
\item for any $\gamma \in (0,\frac{\delta}{2})$, we have $d_Y(A_{\ell}, U^{-1}(-\infty, \gamma])>\delta/2$ for all $\ell$, 
\item for all $x \in A_\ell$ 
 \begin{equation}\label{eq-epsilon3}
0< (c_\ell)^2 - \epsilon_3 \leq |d_x^{S} f|^2 \leq (c_\ell)^2,
 \end{equation}
where $c_\ell = \Lip(f|_{A_\ell})$ and $\epsilon_3 = \epsilon_1 \left(\|T\|(W')\right)^{-1} > 0$, and 
\item letting $\cup_\ell$ denote $\cup_{\ell=1}^N$ henceforth,
\begin{equation}\label{eq-epsilon4}
\|T\|((S'\setminus K^\delta) \setminus \cup_\ell A_\ell) < \epsilon_4,
\end{equation}
where $\epsilon_4 = (1+3\Lambda)^{-2} \epsilon_1$.
\end{itemize}
Furthermore, there exists $b>0$ sufficiently small so that
the open $b/10$-neighborhoods of each $A_{\ell}$ in $Y$, denoted $V_{\ell}$, satisfy:
 \begin{itemize}
 \item each  $V_{\ell}$ is a subset of $W'$,
\item each $V_\ell$ is either a subset of $W$ or of $Y \setminus \ol W$,
\item the $V_{\ell}$ are pairwise disjoint, 
\item $d_Y(V_{\ell}, U^{-1}(-\infty, \gamma])>9b/10$ for all $\ell$, and
\item for $\epsilon_5=\Lambda^{-2} \epsilon_1$,
\begin{equation}
\label{eqn_V_A}
 \|T\|\left(\cup_\ell \ol V_\ell \setminus \cup A_{\ell}\right) < \epsilon_5.
\end{equation} 
(If $\Lambda=0$,  we take $\epsilon_5=1$.)
\end{itemize}
 \end{lemma}

\medskip 

\begin{proof}
Apply Lemmas 3.2 and 6.1 of \cite{Por} to the integral $m$-current $T \llcorner (S' \setminus K^\delta)$ and the restriction of $f$ to $S' \setminus K^\delta$ with  $\epsilon_3 = \epsilon_1 \left(\|T\|(W')\right)^{-1} > 0$, thereby obtaining a sequence of compact sets $A_{\ell} \subset S' \setminus K^\delta$ for $\ell=1,2,\ldots$ such that:
\begin{itemize}
\item Each $A_\ell$ is the bi-Lipschitz image of a compact subset of $\R^m$.
\item The $A_{\ell}$ are pairwise disjoint. 
\item $\cup_{\ell=1}^\infty A_\ell$ has zero co-measure in $S'\setminus K^\delta$ with respect to $\|T\|$.
\item For all $x \in A_\ell$ 
 \begin{equation*}
 0<(c_\ell)^2 - \epsilon_3 \leq |d_x^{S} f|^2 \leq (c_\ell)^2,
 \end{equation*}
where $c_\ell = \Lip(f|_{A_\ell})$. 
\end{itemize}
Since $\|T\|$ is Borel regular and $\|T\|(\partial W)=0$,
we may assume without loss of generality that each $A_\ell$ is either a subset of $W$ or of $Y \setminus \ol W$.

\medskip
Now, 
choose a finite subset of $\{A_{\ell}\}$, call it $A_1, \ldots, A_N$, such that
\begin{equation*}
\|T\|((S'\setminus K^\delta) \setminus \cup_\ell A_\ell) < \epsilon_4,
\end{equation*}
where, for the rest of this proof, $\cup_\ell$ will denote $\cup_{\ell=1}^N$. Let $\gamma \in (0, \frac{\delta}{2})$.
We claim that the distance from any  $A_{\ell}$ to  $U^{-1}(-\infty,\gamma]$ in $Y$ is at least $\delta-\gamma > \frac{\delta}{2}$. 
Let $q \in A_{\ell}$. Since $A_{\ell}$ is disjoint from $K^\delta$ we have $U(q) = u(q) > \delta$. Then since $U$ is 1-Lipschitz,
 if $z \in U^{-1}(-\infty,\gamma]$,
$$d_Y(q,z) \geq |U(q) - U(z)| > \delta-\gamma>\frac{\delta}{2},$$
which proves the claim.   
 
 We know that these finitely many $A_\ell$ are pairwise disjoint and compact; so let $a>0$ be the minimum pairwise distance between them. Let $b$ be a positive real number less than $\min \{a, \frac{\delta}{2}\}>0$  and $V_{\ell}$ be  the open $b/10$-neighborhood of $A_{\ell}$ in $Y$, so $V_1, \ldots, V_N$ are pairwise disjoint
 and their distance to  $U^{-1}(-\infty,\gamma]$   is greater than $9b/10$. Since the $A_{\ell}$ are compact subsets of $W'$, we may shrink $b>0$ if necessary to ensure the $V_{\ell}$ are subsets of $W'$. We can also shrink $b>0$ again to guarantee each $V_\ell$ is either a subset of $W$ or of $Y \setminus \ol W$. Furthermore, since $\|T\|$ is regular, we may shrink $b>0$ if necessary to ensure 
 \begin{equation*}
 \|T\|\left(\cup_\ell \ol V_\ell \setminus \cup A_{\ell}\right) < \epsilon_5.\qedhere
\end{equation*} 
\end{proof}

\medskip

Following as in (6.14)--(6.17) of \cite{Por}, we
are ready to define functions $f_i$ that satisfy the conditions of Proposition \ref{prop-Energy}.

\begin{lemma}
\label{lemma_extension}
Let $\epsilon_1>0$ be given.
Suppose Assumptions \ref{assum_extrinsic} hold.
 Let  $W, W' \subset Y$ be balls about $z_0$ of radius in $(r_0, r_0+1)$ and $(r_0+2, r_0+3)$, respectively,  and $\delta>0$ be given as in Lemma \ref{lem:Kdeltas} (using the given value of $\epsilon_1$). 
 Let $0< \gamma< \delta/2$ be sufficiently small so that Lemma \ref{lem-extensionf} holds
for this choice of $\gamma$ and fixed $f$, i.e., there exists an extension of $f$ by 1 to $S \cup U^{-1}(-\infty,\gamma]$ with Lipschitz constant $\leq 2\Lambda$.  Take $a>0$, $b>0$, and sets $A_\ell$ and $V_\ell$ as in Lemma \ref{lem:setsAandV}. Then there exists a Lipschitz function $\hat f: Y \to \R$ such that
\begin{enumerate}[(a)]
\item $\Lip(\hat f) \leq 1+3\Lambda$,
\item $0 \leq \hat f \leq 1$,
\item $\hat f$ agrees with $f$ on $\cup_\ell A_\ell$ and on $U^{-1}(-\infty,\gamma]$ (in particular, $\hat f|_{U^{-1}(-\infty,\gamma]} = 1$), 
\item $\Lip(\hat f|_{V_\ell}) \leq \Lip(f|_{A_\ell})$ for $\ell=1,\ldots, N$,
\item $\hat f \equiv 1$ on a neighborhood of $K_i$ for all $i$ sufficiently large, and
\item $\spt(\hat f) \subseteq \ol{W'} \subseteq B(z_0,r_0+3)$.
\end{enumerate}
In particular, the functions \[
f_i:=\hat f|_{S_i}: S_i \to \R
\]
are Lipschitz, bounded, have bounded support in $B(z_0,r_0+3)$, and equal $1$ on a neighborhood of $K_i$ (for all $i$ sufficiently large), i.e., $f_i$ is an allowable test function for $\capac_{N_i}(K_i)$.
\end{lemma}

Again, refer to Figures \ref{fig1} and \ref{fig2} for an illustration of some aspects of this setup.

Note that in Lemma \ref{lem-extensionf} we extend $f: S \to \R$ to a function $f: S \cup U^{-1}(-\infty, \gamma] \to \R$. We now construct  $\hat f$ by prescribing its values on $W' \setminus U^{-1}(-\infty, \gamma]$ and extending by zero outside $W'$. We remark that $\hat f$ will not generally be an extension of $f$: the two may differ on $\left( \cup V_\ell \setminus \cup A_{\ell}\right) \cap S$.

\begin{proof}
First, we define for each $\ell$ the standard Lipschitz extension of $f|_{A_\ell}$ to $V_\ell$:
\begin{equation*}
f^\ell(x) = \inf_{a\in A_\ell} \left(f(a) + c_\ell d_Y(a,x)\right),\qquad x \in V_\ell,
\end{equation*}
where, again, $c_\ell= \Lip(f|_{A_\ell})$. Note that $\Lip(f^\ell) = c_\ell$. 
Truncate these functions by defining $\hat{f}^\ell :=  \min\{ f^\ell, 1 \}$, and note  $\Lip(\hat{f}^\ell) \leq c_\ell$ and $0 \leq \hat f^\ell \leq 1$ (recalling $0 \leq f \leq 1$). 

Subsequently, define the function $\hat{f}_1:  \left( \cup_\ell V_\ell\right)\cup U^{-1}(-\infty,\gamma]  \to \R$ by
\begin{equation*}
\hat{f}_1(x) = 
\begin{cases}
\hat{f}^\ell(x)   & \text{ if } x \in V_\ell \\
1  & \text{ if } x \in U^{-1}(-\infty, \gamma],
\end{cases}
\end{equation*}
which satisfies $0 \leq \hat f_1 \leq 1$.

Let us prove that $\Lip(\hat{f}_1) \leq 3\Lambda$. There are only two nontrivial cases. First, if $x\in V_{\ell_1}$, $y \in V_{\ell_2}$, $\ell_1 \neq \ell_2$, there exist $x_0 \in A_{\ell_1}$ and $y_0 \in A_{\ell_2}$ such that $d_Y(x,x_0) < \frac{b}{10}$ and $d_Y(y,y_0) < \frac{b}{10}$. By the triangle inequality, since $d_Y(x_0,y_0) \geq a$, we find $d_Y(x,y) \geq \tfrac{4}{5} b$.
Therefore 
\begin{equation*}
\begin{split}
|\hat{f}_1(x) - \hat{f}_1(y)| &\leq |\hat{f}_1(x) - \hat{f}_1(x_0)| + |\hat{f}_1(x_0) - \hat{f}_1(y_0)| + |\hat{f}_1(y_0) - \hat{f}_1(y)| \\
&\leq c_{\ell_1} d_Y(x,x_0) + \Lip(f) d_Y(x_0, y_0) + c_{\ell_2} d_Y(y_0, y) \\
&\leq \Lip(f)\left(  d_Y(x,x_0) + d_Y(x_0 , y_0) +  d_Y(y_0 , y) \right) \\
&\leq \Lip(f)\left(  \frac{b}{10} + d_Y(x_0 , x) + d_Y(x , y) + d_Y(y , y_0) +  \frac{b}{10} \right) \\
&\leq \Lip(f)\left(  \frac{4b}{10} + d_Y(x , y) \right) \\&\leq \Lip(f) \left( \frac{3}{2} d_Y(x,y) \right)\\
&\leq 3\Lambda d_Y(x , y).
\end{split}
\end{equation*}

Second, assume that $x\in V_{\ell}$, $y \in  U^{-1}(-\infty,\gamma] $, so $d_Y(x,y) \geq \frac{9b}{10}$.  There exists $x_0 \in A_{\ell}$ such that $d_Y(x,x_0) < \frac{b}{10}$ and therefore
\begin{equation*}
\begin{split}
|\hat{f}_1(x) - \hat{f}_1(y)| &\leq |\hat{f}_1(x) - \hat{f}_1(x_0)| + |\hat{f}_1(x_0) - \hat{f}_1(y)|  \\
&\leq c_{\ell_1} d_Y(x, x_0) + \Lip(f) d_Y(x_0, y)  \\
&\leq \Lip(f)\left(  d_Y(x,x_0) + d_Y(x_0, y)  \right) \\
&\leq \Lip(f)\left(  d_Y(x,x_0) + d_Y(x_0, x) + d_Y(x, y) \right) \\
&\leq \Lip(f)\left(  \frac{2b}{10} + d_Y(x, y) \right) \\
&\leq \Lip(f)\left( \frac{11}{9} d_Y(x,y)  \right) \\
&\leq 3\Lambda  d_Y(x,y).\\
\end{split}
\end{equation*}

Consequently,  $\hat{f}_1:  \left( \cup_\ell V_\ell\right)\cup U^{-1}(-\infty,\gamma]  \to \R$ has Lipschitz constant at most $3\Lambda$. We extend $\hat{f}_1$ to a Lipschitz function, denoted in the same way, on all $Y$ in the standard way (with the same Lipschitz constant): 
$$\hat f_1(y) := \inf_{x\in  \left( \cup_\ell V_\ell\right)\cup U^{-1}(-\infty,\gamma] } \left(\hat f_1(x) + \Lip(\hat f_1) d_Y(x,y)\right),$$
truncate at the value $1$, and call the result $\hat f_2$. Note $\Lip(\hat f_2) = \Lip (\hat f_1)\leq 3\Lambda$.

However, $\hat f_2$ will not have bounded support, so we will modify $\hat f_2$ using a cutoff function while ensuring claims (a)--(f) will hold. Let $0\leq \rho \leq 1$ be a Lipschitz function on $Y$ that is identically one on $\ol W$ and is supported in $W'$. Since the radii of $W$ and $W'$ differ by more than 1, we may assume without loss of generality that
\begin{equation}
\label{eqn_MLip}
\Lip(\rho) \leq 1.
\end{equation}

We claim  $\hat f =\rho \hat f_2$ is the desired function. The function $\hat f$ clearly has bounded support in $\ol{W'}$ and is bounded between 0 and 1, i.e. (b) and (f) hold. We see $\hat f$ is Lipschitz as well:
\begin{align*} 
| \hat f(x) -  \hat f(y)  |  = &
| \rho(x) \hat f_2(x) -   \rho(y) \hat f_2(y)  | \\
\leq  &  | \rho (x) -   \rho (y) |  |  \hat f_2(x) |  +  | \hat f_2(x) - \hat f_2(y)  |  |  \rho (y)  | \\
\leq  &  \left(  \Lip(\rho) +   \Lip(\hat f_2) \right) d_Y(x , y )\\
\leq &   \left(1 +   3\Lambda \right) d_Y(x , y ),
\end{align*}
having used \eqref{eqn_MLip}. That is, (a) holds. 

Next, we show (c): first suppose $x \in A_\ell$. If $A_\ell \subset W$, then $\rho(x)=1$ and it follows that
\[
\hat f(x)= \rho(x)\hat f_2(x)=f^\ell(x)=f(x).
\]
Otherwise, $A_\ell \subset Y \setminus \ol W$, which implies $f|_{A_\ell}=0$, as $W$ contains the support of $f$. Then $c_\ell=0$, so $\hat f_2(x)=f^\ell(x)=0$. Then $\hat f=\rho \hat f_2$ and $f$ both vanish on $A_\ell$. Next, suppose $x \in U^{-1}(-\infty,\gamma]$. Since $W \supset U^{-1}(-\infty,\gamma]$ and $\rho \equiv 1$ on $W$, it follows that $\hat f(x)=\rho(x)\hat f_2(x)=1=f(x)$.

We show (d): If $V_\ell$ is a subset of $W$, then $\rho \equiv 1$ on $V_\ell$, so $\hat f|_{V_{\ell}}=\rho \hat f_2|_{V_{\ell}} = \hat f_2|_{V_\ell} = f^\ell$, whose Lipschitz constant is at most $c_\ell = \Lip(f|_{A_\ell})$. On the other hand if $V_\ell$ is a subset of $Y \setminus \ol W$, then $\rho \hat f_2|_{V_{\ell}}=0$. But since $A_\ell \subset Y \setminus \ol W$ and $\spt(f) \subseteq W$, we have $f|_{A_\ell}=0$, which implies $f^\ell=0$. So in this case as well, $\Lip(\hat f|_{V_{\ell}}) \leq \Lip(f|_{A_\ell})$ (both are zero).

To show (e), restrict to $i$ sufficiently large so that $\alpha_i < \gamma$. Then $U^{-1}(-\infty, \alpha_i] \subset U^{-1}(-\infty, \gamma)$. By definition $\hat f_2 \equiv 1$ on $U^{-1}(-\infty, \gamma]$. Since $U^{-1}(-\infty, \gamma] \subset W$ and $\rho \equiv 1$ on $W$, it follows that $\rho \hat f_2 \equiv 1$ on $U^{-1}(-\infty, \gamma]$ as well. In particular, $\hat f=\rho \hat f_2$ is identically 1 on the neighborhood $U^{-1}(-\infty, \gamma)$ of $K_i=U^{-1}(-\infty, a_i] \cap S_i$, i.e. (e) holds.

Thus $\hat f=\rho \hat f_2$ satisfies (a)--(f).
\end{proof}

We now prove Proposition \ref{prop-Energy} by establishing the energy estimate \eqref{USC_energy_with_eps}, which will later be sufficient to obtain  Theorem \ref{thm_extrinsic}.

\begin{proof}[Proof of Proposition \ref{prop-Energy}]
Suppose Assumptions \ref{assum_extrinsic} hold, and let $\epsilon>0$ be given.
 Set the value $\epsilon_1 = \frac{\epsilon}{4}$ to be used in the Lemmas \ref{lem:Kdeltas}--\ref{lemma_extension}. Note that from Assumptions \ref{assum_extrinsic} and Lemma \ref{lemma_mass_convergence}, for every closed, bounded set $C \subset Y$, we have
\begin{equation}
\label{limsup}
\limsup_{i \to \infty} \|T_i\|(C) \leq \|T\|(C).
\end{equation}

Take $\delta>0$ and balls $W$ and $W'$ provided by Lemma \ref{lem:Kdeltas}. Let $\gamma \in \left(0, \frac{\delta}{2}\right)$ be sufficiently small so that Lemma \ref{lem-extensionf} holds. Take sets $A_\ell$ and $V_\ell$ and constants $c_\ell$ provided by Lemma \ref{lem:setsAandV}. Finally, using Lemma \ref{lemma_extension}, obtain  $\hat f: Y \to \R$, and $f_i= \hat f|_{S_i}$ so that 
$f_i \in \LipB(S_i)$, $0 \leq f_i \leq 1$, with $f_i \equiv 1$ on a neighborhood of $K_i$ (for $i$ sufficiently large) such that
\[ 
\Lip(f_i) \leq 1+ 3\Lambda, \quad \spt(f_i) \subset \overline{W}' \subset B(z_0,r_0+3).
\]

Now we begin to consider  $E_{N_i}(f_i)$:
\begin{align}
 E_{N_i}(f_i) &=  \int_{Y} |d_x^{S_i} f_i|^2 d\|T_i\| \nonumber \\
&= \int_{Y\setminus W'} |d_x^{S_i} \hat f|^2 d\|T_i\| +
 \int_{\cup_\ell V_\ell} |d_x^{S_i} \hat f|^2 d\|T_i\| + \int_{W' \setminus\cup_\ell V_\ell} |d_x^{S_i} \hat f|^2 d\|T_i\|.\label{ineq_3}
\end{align}
The first term is zero, since $\hat f$ vanishes outside $W'$. For the second term in \eqref{ineq_3}, by Lemma \ref{lemma_extension}(d) and \eqref{eq-epsilon3} we get
 \begin{align*}
 \int_{\cup_\ell V_\ell} |d_x^{S_i} \hat f|^2 d\|T_i\| &\leq \sum_{\ell=1}^N \Lip(\hat f|_{V_\ell})^2 \|T_i\|(V_\ell)\\
&\leq \sum_{\ell=1}^N (c_\ell)^2 \|T_i\|(V_\ell).
\end{align*}
We replace $V_{\ell}$ with its closure in the above and take the limsup to obtain, by \eqref{limsup}, \eqref{eqn_V_A}, and \eqref{eq-epsilon3}:
\begin{align}
\limsup_{i \to \infty} \int_{\cup_\ell V_\ell} |d_x^{S_i} \hat f|^2 d\|T_i\| &\leq  \limsup_{i \to \infty} \sum_{\ell=1}^N (c_\ell)^2 \|T_i\|( \ol V_\ell)\nonumber\\
&\leq \sum_{\ell=1}^N (c_\ell)^2 \|T\|( \ol V_\ell)\nonumber\\
&\leq \sum_{\ell=1}^N (c_\ell)^2 \|T\|( A_\ell)+ \Lambda^2 \epsilon_5\nonumber \\
    &\leq\   \int_{\cup_\ell A_\ell} \left(|d_x^{S} f|^2 +\epsilon_3 \right) d\|T\| +  \epsilon_1\nonumber\\
&\leq  E_{N_\infty}(f) +  \|T\|(W')\epsilon_3+ \epsilon_1 \nonumber \\
&\leq E_{N_\infty}(f) + \epsilon_1 + \epsilon_1.\label{eqn_second_term}
\end{align}

\bigskip 
For the third term in \eqref{ineq_3}  we may omit integration over the open set $U^{-1}(-\infty, \gamma)$, since $\hat f \equiv 1$ there. It follows that:
$$ \int_{W' \setminus\cup_\ell V_\ell} |d_x^{S_i} \hat f|^2 d\|T_i\| \leq \Lip(\hat f)^2 \|T_i\|\left(\ol{W'} \setminus \left(\cup_\ell V_\ell \cup U^{-1}(-\infty,\gamma)\right)\right).$$
The set on the right is closed and bounded, so by \eqref{limsup}:
\begin{align*}
\limsup_{i\to \infty} \int_{W' \setminus\cup_\ell V_\ell} |d_x^{S_i} \hat f|^2 d\|T_i\| &\leq \Lip(\hat f)^2 \|T\|\left(\ol{W'} \setminus \left(\cup_\ell V_\ell \cup U^{-1}(-\infty,\gamma)\right)\right).
\end{align*}
Now, by considering intersections and set subtractions with $S'$ and with $K^\delta$, we find
$$\left(\ol{W'} \setminus U^{-1}(-\infty,\gamma)\right)\setminus\cup_\ell V_\ell \subseteq (\ol{W'} \setminus S') \cup \left((S' \setminus K^\delta)\setminus \cup V_{\ell} \right) \cup (K^\delta \setminus K).$$
Thus, using \eqref{eq-epsilon4} and \eqref{eq-epsilon2},
\begin{align}
\limsup_{i\to \infty} \int_{W' \setminus\cup_\ell V_\ell} &|d_x^{S_i} \hat f|^2 d\|T_i\| \nonumber \\
&\leq \Lip(\hat f)^2\left[ \|T\|\left(\ol{W'} \setminus S'\right) +\|T\|\left((S' \setminus K^\delta)\setminus \cup V_{\ell} \right)+\|T\|\left(K^\delta \setminus K\right)
\right]\nonumber\\
&\leq \Lip(\hat f)^2\left[ \|T\|\left(W' \setminus S'\right) +\|T\|(\partial W') +\epsilon_4 + \epsilon_2
\right]. \label{eqn_third_term}
\end{align}
Note that $\|T\|\left(W' \setminus S'\right) \leq \|T\|\left(Y \setminus S\right)$, where we recall $S= \set(T)$. It is shown in \cite[Theorem 4.6]{AK_cur} that an integral current's measure is concentrated on its canonical set. The same goes for locally integral currents, i.e. $\|T\|\left(Y \setminus S\right)=0$. The next term, $\|T\|(\partial W')$, vanishes by  \eqref{eqn_bdry_W'}. 

Combining \eqref{ineq_3}, \eqref{eqn_second_term}, and \eqref{eqn_third_term}, we have (using the definition of $\epsilon_2$ and $\epsilon_4$):
\begin{align*}
\limsup_{i \to \infty} E_{N_i}(f_i) &\leq   E_{N_\infty}(f) + 2\epsilon_1 + 
\left(1+3\Lambda\right)^2 (\epsilon_4 + \epsilon_2)\\
&= E_{N_\infty}(f) +  4\epsilon_1.
\end{align*}
Since $\epsilon=4\epsilon_1>0$, we have concluded the proof of Proposition \ref{prop-Energy}.
\end{proof}

\begin{proof}[Proof of Theorem \ref{thm_extrinsic}]
We establish  \eqref{USC_extrinsic} first. By the definition of capacity, given $\epsilon>0$, choose $f \in \LipB(S)$, with $0 \leq f \leq 1$ and $f \equiv 1$ on a neighborhood of $K$, such that
\begin{equation}\label{eqn_E_cap}
E_{N_\infty}(f) \leq \gamma_m\capac_{N_\infty}(K) + \epsilon.
\end{equation}
We use this choice of $f$ in Assumptions \ref{assum_extrinsic}. 
We  apply Proposition \ref{prop-Energy} to see there exists a sequence $f_i \in \LipB(S_i)$, $0 \leq f_i \leq 1$, with $f_i \equiv 1$ on a neighborhood of $K_i$ (for $i$ sufficiently large), $\Lip(f_i) \leq 1+ 3\Lambda$
such that
\begin{equation*}
\limsup_{i \to \infty} E_{N_i}(f_i) \leq E_{N_\infty}(f) + \epsilon.
\end{equation*}
Combining this with \eqref{eqn_E_cap} and the definition of capacity,  \eqref{USC_extrinsic}  follows, since $\epsilon$ was arbitrary.

\medskip

Let us now establish \eqref{USC_energy}.
By Proposition \ref{prop-Energy}, we may assume that for every $j \in \mathbb{N}$ we have a sequence $\left(f^{1/j}_i\right)_i$ such that 
\[f
_i^{1/j} \in \LipB(S_i), \quad 0 \leq f_i^{1/j} \leq 1, \quad \Lip(f_i^{1/j}) \leq 1+ 3\Lambda, \quad \spt(f_i^{1/j}) \subseteq B(z_0,r_0+3),
\]
with $f_i^{1/j} \equiv 1$ on a neighborhood of $K_i$ (for $i$ sufficiently large, depending on $1/j$), such that
\begin{equation*}
\limsup_{i \to \infty} E_{N_i}(f_i^{1/j}) \leq E_{N_\infty}(f) + 1/j.
\end{equation*}
Then we may construct a monotonically increasing sequence $n: \mathbb{N} \to \mathbb{N}$ such that $n_1 = 1$ and for all $j \in \mathbb{N}\setminus\{1\}$ it holds that for all $i \geq n_j$, $f_i^{1/j} \equiv 1$ on a neighborhood of $K_i$ and 
\[
E_{N_i}(f_i^{1/j}) \leq E_{N_\infty}(f) + 2 / j.
\]
We then define a new non-decreasing sequence $m : \mathbb{N} \to \mathbb{N}$ by
\[
m_{j} = \max\{ k \ |\  n_k \leq j \}.
\]
Note that the maximum is well-defined as $(n_k)$ is monotonically increasing. Moreover, $m_j \to \infty$ as $j \to \infty$ since for every $M_0 \in \mathbb{N}$, if we define $j_0 := n_{M_0}$, we have $m_{j_0} \geq M_0$ (in fact $m_{j_0} = M_0$ since $n: \mathbb{N} \to \mathbb{N}$ is monotonically increasing).
We finally define $f_i$ as
\[ 
f_i := f^{1/m_i}_i.
\]
 Since by construction $i \geq n_{m_i}$ for all $i \in \mathbb{N}$, we know that for all $i \in \mathbb{N} \setminus\{1\}$ $f^{1/m_i}_i\equiv 1$ on a neighborhood of $K_i$, and 
\[ 
E_{N_i}(f_i)= E_{N_i}\left(f^{1/m_i}_i\right) \leq E_{N_\infty}(f) + 2/m_i
\]
whereas $m_i \to \infty$ as $i \to \infty$. Thus, the $f_i$'s satisfy all the properties mentioned in the theorem.
So,  \eqref{USC_energy} has been established.
\end{proof}

\bigskip

Now we may prove the first main theorem, Theorem \ref{thm_balls}.

\begin{proof}[Proof of Theorem \ref{thm_balls}]
We divide the proof in two cases. Assume first that $K=\ol B(p, r) \neq X$. Given $\epsilon > 0$, take a  function $f \in \Lip_B(X)$, $0\leq f \leq 1$, with $f \equiv 1$ in a neighborhood $O$ of $K$ and
\begin{equation}
E_N(f) \leq \gamma_m\capac_N(K) + \epsilon. \label{eqn_E_N_eps}
\end{equation}
We aim to fulfill Assumptions \ref{assum_extrinsic}, so that we can apply Theorem \ref{thm_extrinsic}. Choose $r_0 > 3 \diam(K)+d(K, X \setminus K)$ sufficiently large so that $K \subseteq \spt(f) \subseteq B(p,r_0)$, 
and
\begin{equation}
\label{eqn_B_X_K}
B(p,r_0) \cap (X \setminus K) \neq \emptyset.
\end{equation}

Using Definition \ref{def_pointed_F}, choose $R>r_0+4$ such that 
\[
N_i \llcorner B(p_i,R) \to N \llcorner B(p,R) \quad  \textrm{in the $\VF$ sense. }
\]
Let $S_i = \set(T_i \llcorner B(p_i,R)) \subseteq X_i$ and $S=\set(T \llcorner B(p,R))\subseteq X$. Since $p \in X = \set(T)$, it follows that $p \in \set(T\llcorner B(p,R)) =  S$.  Thus $S \neq \emptyset$ and $T \llcorner B(p,R) \neq 0$. It is straightforward to see that $K \subseteq S$. That $K$ is a proper subset of $S$ follows from \eqref{eqn_B_X_K}.

By Theorem \ref{thm_embedding} there exists a $w^*$-separable Banach space $Y$ and distance-preserving maps  $\varphi_i: S_i \to Y$ and $\varphi: S \to Y$ 
such that the integral currents $\varphi_{i\#}(T_i \llcorner B(p_i,R))$ converge to $\varphi_{\#}(T \llcorner B(p,R))$ in the flat $d_Y^F$ sense (and therefore in the weak sense), the masses converge,
\begin{equation}
\M(\varphi_{i\#}(T_i \llcorner B(p_i,R))) \to \M(\varphi_{\#}(T \llcorner B(p,R))), \label{eqn_masses_converge}
\end{equation}
and $\varphi_i(p_i) \to \varphi(p)$ as $i \to \infty$.

 Let $u(x)=d(x,p)-r$ for $x \in X$, which satisfies $\{u\leq 0\} = K$ and has $\Lip(u)=1$. We claim $u$ is a defining function for $K$ as in \eqref{defining_function}. Let $\tilde O$ be any open subset of $X$ containing $K$. For $x \in X \setminus \tilde O$, we have $u(x)>0$ and $d(x,K)>0$ and $u$ and $x \mapsto d(x, K)$ are continuous, positive functions on $X \setminus \tilde{O}$. In particular, $\frac{u(x)}{d(x,K)}$ has a lower bound $\beta_0>0$ on the compact set $(X \setminus \tilde O) \cap \ol B(p,r')$, for $r'>r$. On the other hand, for $x \in (X \setminus \tilde O) \setminus \ol B(p,r')$, we have
 $$\frac{u(x)}{d(x,K)} \geq \frac{u(x)}{d(x,p)} = \frac{d(x,p)-r}{d(x,p)} \geq \frac{r'-r}{r'}>0,$$
 since $z \mapsto \frac{z-r}{z}$ is increasing for $z>0$. Choosing $\beta = \min\left(\beta_0, \frac{r'-r}{r'}\right)>0,$ we conclude $u(x)$ satisfies conditions \eqref{defining_function} of a defining function. (Note: if $(X \setminus \tilde O) \cap \ol B(p,r')$ or $(X \setminus \tilde O) \setminus \ol B(p,r')$ is empty, the proof is easily modified, or trivial if both sets are empty.)

Since $\varphi$ is a distance-preserving map, $u \circ \varphi^{-1}:\varphi(S) \to \R$ is equal to $d_Y(\varphi(p), \cdot)-r$. It is elementary to verify that the 1-Lipschitz extension $U$ of $u \circ \varphi^{-1}$ to $Y$ is simply given by $d_Y(\varphi(p), \cdot)-r$.

To fulfill Assumptions \ref{assum_extrinsic}, take the Banach space $Y$, 
the nonzero integral current $\varphi_{\#}(T \llcorner B(p,R))$ on $Y$, whose canonical set is $\varphi(S)$, 
the weakly converging sequence of integral currents $\varphi_{i\#}(T_i \llcorner B(p_i,R))$ on $Y$, whose canonical sets are $\varphi_i(S_i)$, the nonempty compact set $\varphi(K) \subsetneq \varphi(S)$ in $Y$, the defining function $u \circ \varphi^{-1}$ of $\varphi(K)$ with standard 1-Lipschitz extension $U:Y \to \R$, 
the sequence $\alpha_i=d_Y(\varphi_i(p_i), \varphi(p))\geq 0$ (which converges to 0 as $i \to \infty$), the Lipschitz function $f \circ \varphi^{-1}: \varphi(S) \to \R$, the point $z_0=\varphi(p)$, and the value $r_0$ above. 
Let $\hat N= (\varphi(S), d_Y, \varphi_{\#}(T \llcorner B(p,R)))$ and analogously,  $\hat N_i= (\varphi_i(S_i), d_Y, \varphi_{i\#}(T_i \llcorner B(p_i,R)))$. 

A few hypotheses in Assumptions \ref{assum_extrinsic} require verification in order to apply Theorem \ref{thm_extrinsic}. First, we claim that
\begin{equation}
\label{eqn_r0}
r_0 >3\diam(\varphi(K)) + d_Y(\varphi(K), \varphi(S) \setminus \varphi(K)).
\end{equation}
By our choice of $r_0$ and the fact that isometries preserve diameter, it suffices to show 
\begin{equation}\label{eq-disKtoCom}
  d(K, X \setminus K) \geq d_Y(\varphi(K), \varphi(S) \setminus \varphi(K)).  
\end{equation}
Given $\eta>0$, there exists $k \in K$ and $x \in X \setminus K$ such that
$$d(k,x) \leq d(K, X \setminus K ) + \eta.$$
Then $\varphi(k) \in \varphi(K)$, and we claim that $x \in S$ (if $\eta$ was chosen sufficiently small). By the triangle inequality,
\begin{align*}
 d(p,x) &\leq  d(p,k) +  d(k,x)\\
&\leq \diam(K) + d(K, X \setminus K) + \eta ,
\end{align*}
which is less than $R$ if $\eta$ is sufficiently small. Then $x \in B(p,R)$. Since $x \in X = \set(T),$ we have $x \in \set (T \llcorner B(p,R))= S$. Then $\varphi(x) \in \varphi(S) \setminus \varphi(K)$, so
$$d_Y(\varphi(K), \varphi(S) \setminus \varphi(K)) \leq d_Y(\varphi(k), \varphi(x)) = d(k,x) \leq d(K, X \setminus K ) + \eta.$$
Since $\eta>0$ can be arbitrarily small, the proof of \eqref{eq-disKtoCom} is complete. So the proof of  claim \eqref{eqn_r0} is complete.

Second, $f \circ \varphi^{-1}$ is clearly Lipschitz, bounded between 0 and 1, with $f \circ \varphi^{-1} \equiv 1$ in a neighborhood of $\varphi(K)$. Since $\spt(f) \subseteq B(p,r_0)$ and $\varphi$ is a distance-preserving map, we have $\spt(f \circ \varphi^{-1}) \subseteq B(\varphi(p),r_0) = B(z_0,r_0)$.

Third, the mass convergence hypothesis holds by \eqref{eqn_masses_converge}, with $V$ taken to be any ball about $z_0$ of radius greater than $R$.

Now, Assumptions \ref{assum_extrinsic} hold, so by Theorem \ref{thm_extrinsic}, for each $i$ sufficiently large, there exists a Lipschitz function $f_i:\varphi_i(S_i) \to \R$, $0 \leq f_i \leq 1$, with $f_i\equiv 1$ on a neighborhood of $K_i$  (where $K_i=U^{-1}(-\infty,\alpha_i] \cap \varphi_i(S_i)$),
and $\spt(f_i) \subseteq B(\varphi(p), r_0+3)$, such that
\begin{equation}
\label{limsup_E}
\limsup_{i \to \infty} E_{\hat N_i}(f_i) \leq E_{\hat N}(f \circ \varphi^{-1}).
\end{equation}

Consider $f_i \circ \varphi_i:S_i \to \R$, which is Lipschitz, bounded between 0 and 1, equalling 1 on a neighborhood of $\varphi_i^{-1}(K_i)$. To control the support, we have that for $i$ large, $d_Y(\varphi_i(p_i), \varphi(p)) < 1$. From this it follows $\spt(f_i) \subseteq B(\varphi_i(p_i), r_0+4)$, and so
$$\spt(f_i \circ \varphi_i) \subseteq B(p_i, r_0+4) \subseteq B(p_i,R).$$
Thus, we may extend $f_i \circ \varphi_i$ by 0 on $X_i \setminus B(p_i,R)$ to produce a Lipschitz function on $X_i$ with the same Dirichlet energy; call it $\hat f_i$, which is a valid test function for the capacity of $K_i$. 

Using \eqref{limsup_E} on the fourth line below and \eqref{eqn_E_N_eps} on the last line,
\begin{align*}
\limsup_{i \to \infty} \gamma_m\capac_{N_i}(\varphi_i^{-1}(K_i)) &\leq \limsup_{i \to \infty} E_{N_i} (\hat f_i)
 \\
&= \limsup_{i \to \infty} E_{N_i \llcorner B(p_i,R)} (f_i \circ \varphi_i) \\
&= \limsup_{i \to \infty} E_{\varphi_{i\#}(N_i \llcorner B(p_i,R))} (f_i) \\
&\leq E_{\varphi_{\#}(N \llcorner B(p,R))} (f \circ \varphi^{-1}) \\
&= E_{N \llcorner B(p,R)} (f) \\
&= E_{N} (f)\\
&\leq \gamma_m\capac_N(K) + \epsilon.
\end{align*}

Since $\epsilon>0$ was arbitrary, the proof will now follow from the monotonicity of capacity by showing $\ol B(p_i,r) \subseteq \varphi_i^{-1}(K_i)$.  To verify this, let $z \in \ol B(p_i,r)$. We must show $\varphi_i(z) \in K_i$, i.e., $U(\varphi_i(z))\leq \alpha_i$:
\begin{align*}
U(\varphi_i(z)) &= d_Y(\varphi(p), \varphi_i(z)) - r\\
&\leq d_Y(\varphi(p), \varphi_i(p_i)) + d_Y(\varphi_i(p_i), \varphi_i(z)) - r\\
&= \alpha_i + d_i(z,p_i) -r \leq \alpha_i.
\end{align*}

The proof of the theorem is complete in the case $K=\ol B(p,r) \neq X$.

\medskip 
To complete the proof, we consider the case $K=\ol B(p, r) = X$. In particular, $X$ is bounded and the function $f \equiv 1$ on $X$ is a valid test function for the capacity, i.e. $\capac_N(\ol B(p, r))=0$. 

Now, choose $R>r$ so that $N_i \llcorner B(p_i,R)$ and $N_i \llcorner B(p_i,R+1)$ converge in the $\VF$ sense as integral current spaces to $N \llcorner B(p,R)=N=N \llcorner B(p,R+1)$. By mass convergence,
$$\|T_i\|(A(p_i,R,R+1)) \to 0 \text{ as } i \to \infty,$$
where $A(p_i,R,R+1)$ is the annulus $B(p_i,R+1) \setminus B(p_i,R)$.
Let $f_i$ be a function on $X_i$ that equals 1 on $\ol B(p_i,R)$, 0 on $X_i \setminus B(p_i,R+1)$ and has $\Lip(f_i) \leq 1$. (Such a function can easily be constructed as a radial function of $d_{i}(p_i, \cdot)$.) Then we have
\begin{align*}
\capac_{N_i}(\ol B(p_i,r)) &\leq \capac_{N_i}(\ol B(p_i,R))\\
&\leq\frac{1}{\gamma_m} \int_{X_i} |d_x^{X_i} f_i|^2 d\|T_i\|(x)\\
&\leq\frac{1}{\gamma_m} \Lip(f_i)^2 \|T_i\|(A(p_i,R,R+1)).
\end{align*}
It follows that $\limsup_{i \to \infty}\capac_{N_i}(\ol B(p_i,r))=0$, completing the proof.
\end{proof}

We conclude this section by proving the other main result, Theorem \ref{thm_sublevel}.

\begin{proof}[Proof of Theorem \ref{thm_sublevel}]
Recall that by definition each $T_i$ is a locally integral current defined on  $(\overline{X}_i, \overline{d}_i)$. So, we can apply Theorem 1.1 in \cite{LW} to the currents $T_i$ and points $p_i \in X_i$. The hypothesis
$$\sup_{i \in \mathbb N} \Big( \|T_i\|(B(p_i,r))+ \|\partial T_i\| (B(p_i,r))\Big) < \infty$$
for each $r>0$ holds by \eqref{eqn_bdry_mass_bound} and the hypothesis of pointed $\VF$ convergence. Thus, by Theorem 1.1 in \cite{LW}, there exist a subsequence of $N_i$ (note: in this proof we will not relabel subsequences), a complete metric space $(Z,d_Z)$, a point $z\in Z$, and distance-preserving maps  $\varphi_i: \overline {X}_i \to Z$ such that $\varphi_i(p_i)\to z$ in $Z$ and $\varphi_{i\#} (T_i) \to T'$ in the local flat topology, for some locally integral current $T'$ on $Z$ of dimension $m$. 

We point out that $Z$ can be taken to be a $w^*$-separable Banach space $Y$: first, recall that integral current spaces are separable \cite[Remark 2.36]{SW}, and the same goes for local integral current spaces. The construction of $Z$ in \cite{LW} comes directly from Proposition 5.2 in \cite{We2011}. There, $Z$ is constructed as the completion of a countable union of the $X_i$ and is therefore separable. Thus, we can apply Kuratowski's embedding theorem and replace $Z$ with $Y=\ell^{\infty}(Z)$, a $w^*$-separable Banach space. We then assume the embeddings $\varphi_i$ and $\varphi$ are into $Y$, that $T'$ lives on $Y$, and that $z \in Y$.

Let $N'=(\set(T'), d_Y, T'\llcorner \overline{ \set(T')})$. If we show that $N_i \to N'$ in the pointed $\mathcal{VF}$ sense with respect to $p_i \in X_i$ and $z \in \overline{\set(T')}$, 
then by uniqueness of pointed $\F$ limits (Proposition \ref{prop_pointed_limit}), we would get that $N\cong N'$.  Once this is done, the result will follow by applying Theorem \ref{thm_extrinsic}. 

Let
$$G= \{r > 0 \;:\; N_i \llcorner B(p_i,r) \toF N \llcorner B(p,r) \text{ with } p_i \to p \text{ as } i \to \infty\}.$$
By Definition \ref{def_pointed_F}, $G$ is unbounded. Using the slicing argument in \cite[Lemma 4.1]{Sor}, we pass to a subsequence so that $\R^+ \setminus G$ has measure zero. Applying a similar argument in $Y$, we may pass to a further subsequence and replace $G$ with a subset, still with $\R^+ \setminus G$ having measure zero, such that
\begin{equation}
\label{eqn_T_prime}
\varphi_{i\#}(T_i \llcorner B(p_i,r)) \to T' \llcorner B(z,r)
\end{equation}
as integral currents in the flat sense in $Y$ for all $r \in G$.

We now verify that $z \in \overline{\set(T')}$ using  Theorem 2.9 in \cite{HLP} as follows. Let $r_1 < r_2$ belong to $G$.
Take $N_i \llcorner B(p_i, r_2)$ as ``$M_i$'' and $B(p_i,r_2)$ as ``$V_i$'' in the theorem. Then
$M_i \llcorner V_i \toF N \llcorner B(p,r_2)$ with $p_i \to p$. Thus, condition $(1)$ of that theorem is satisfied, with $N \llcorner B(p,r_2)$ playing the role of ``$N_\infty$'' and $p$ as ``$x_\infty$.''
Condition $(2)$  follows, taking ``$\delta$'' as $r_1$.  
Since we also have  $M_i=N_i \llcorner B(p_i, r_2)  \to  N' \llcorner B(z,r_2)$ in the flat sense in $Y$, with $\varphi_i(p_i) \to z$,  Theorem 2.9 in \cite{HLP} guarantees a subsequence such that $\varphi_i(p_i)$ converges to some $x' \in \overline{\set(T' \llcorner B(z,r_2))}$  in $Y$.   But we know that $\varphi_i(p_i) \to z$ in $Y$. Hence, $z=x'$, and $x' \in  \overline{\set(T')}$.

Finally, we prove that  $N_i \to N'$ in the pointed $\mathcal{VF}$ sense with respect to $p_i \in X_i$ and $z \in \overline{\set(T')}$. Let $r_0>0$. There exists $r \geq r_0$ with $r\in G$. Thus, from \eqref{eqn_T_prime}, $N_i \llcorner B(p_i,r) \toF N' \llcorner B(z,r)$. Since $\varphi_i(p_i) \to z$, we have shown the claim.

We now apply Theorem \ref{thm_extrinsic} to the Banach space $Y$, 
the nonzero integral current $\varphi_{\#}(T)$ on $Y$, where $\varphi: X \to \set(T')$ is the isometry given by the previous paragraph and Proposition \ref{prop_pointed_limit}, the weakly converging sequence of integral currents $\varphi_{i\#}(T_i)$ on $Y$ to $T'=\varphi_{\#}(T)$, the point $z$, the nonempty compact set $\varphi(K) \subsetneq \varphi(X)$, the defining function $u \circ \varphi^{-1}$ of $\varphi(K)$ with standard 1-Lipschitz extension $U:Y \to \R$. 
Let $N'_i= (\varphi_i(X_i), d_Y, \varphi_{i\#}(T_i))$
and $K_i'=U^{-1}(-\infty,\alpha_i] \cap \varphi_i(X_i)$.
Thus,
\begin{equation*}
\limsup_{i \to \infty} \capac_{N'_i}(K'_i) \leq \capac_{N'}(\varphi(K)).
\end{equation*}
Since $N'_i\cong N_i$ and $N' \cong N$ as local integral current spaces, via the distance preserving maps $\varphi_i$ and $\varphi$, respectively, we have by definition of $K_i$: 
$\capac_{N'_i}(K'_i) = \capac_{N_i}(\varphi_i^{-1}(K'_i)) = \capac_{N_i}(K_i)$, and analogously,  
$\capac_{N'}(\varphi(K))=\capac_{N}(K)$, which proves the theorem. 
\end{proof}

\section{Examples}
\label{sec_examples}

In this section we give examples to demonstrate that a) the capacity upper semicontinuity can be strict, i.e., the capacity can jump up in a limit (Examples 1--3), and b) \emph{volume-preserving} convergence is necessary to guarantee upper semicontinuity (Example 4). We find there are essentially two independent reasons for the upper semicontinuity phenomenon. First, even under smooth convergence on compact sets the capacity of a set can jump up due to  non-uniform control at infinity, e.g. a change in the end geometry of the manifold.  Second, under $\VF$-convergence the capacity can also jump up, even with uniform control on the geometry at infinity.

\medskip

\paragraph{\emph{Example 1: Transition from cylindrical to Euclidean end geometry}}

Consider rotationally symmetric smooth Riemannian metrics on $\R^n$, $n \geq 3$ of the form
$$g_i = ds^2 + f_i(s)^2 d\sigma^2,$$
where each $f_i:[0,\infty) \to \R$, $i=1,2,\ldots$ is smooth, with $f_i(0)=0$, $f_i(s)>0$ for $s>0$, and $d\sigma^2$ is the standard metric on the unit $(n-1)$-sphere. If we assume
$$f_i(s) = \begin{cases}
s, & 0 \leq s \leq i\\
i+1,& s \geq i+1
\end{cases},$$
then the corresponding Riemannian manifold $(\R^n, g_i)$ is isometric to a Euclidean ball for $s \leq i$ and to a cylinder (sphere-line product) of radius $i+1$ for $s \geq i+1$. The capacity of every compact set in $(\R^n, g_i)$ is zero, due to the cylindrical end (explained below). However, this sequence of Riemannian manifolds converges smoothly on compact sets, and hence in the pointed $\VF$ sense, to Euclidean space (where of course there exist compact sets of positive capacity). This example shows we cannot expect the capacity to behave continuously even for smooth local convergence.

To verify that the capacity vanishes identically with respect to $g_i$, given $i$, consider a radial  Lipschitz function $\varphi_L(s)$ on $\R^n$ with
$$\varphi_L(s)=\begin{cases}
1, & s \leq L\\
2-\frac{s}{L}, &L < s \leq 2L\\
0, & 2L < s
\end{cases}$$
for a parameter $L$. Taking any $L > i+1$, we have
$$\int_{\R^n} |\nabla \varphi_L|^2 dV_{g_i} = \omega_{n-1}\int_L^{2L} \frac{1}{L^2}ds = \frac{\omega_{n-1}}{L},$$
which can be made arbitrarily small by taking $L$ large. Moreover, by taking $L$ large, we can arrange $\varphi_L=1$ on any compact set.

\medskip

\paragraph{\emph{Example 2: Formation of a new end}}
Let $M=\R \times S^2$ be equipped with a rotationally symmetric Riemannian metric
$$g= ds^2 + f(s)^2 d\sigma^2,$$
where $f>0$ is a smooth, even function. Further, assume $f^{-2}$ is integrable on $\R$. Let $K$ be the compact subset $\{0\} \times S^2$. We compute the capacity of $K$ in $(M,g)$ as follows.

It is elementary to verify that the function
$$\psi(s) = \begin{cases}
\int_0^s f(r)^{-2} dr & s > 0\\
\int_s^0 f(r)^{-2} dr & s < 0
\end{cases}$$
is $g$-harmonic on $M \setminus K$, equalling zero on $K$ and approaching a positive constant $C=\int_0^\infty f(r)^{-2} dr$ at $\pm \infty$. In particular, $\varphi= 1-\frac{1}{C}\psi$ is a minimizer for the capacity of $K$. (Although $\varphi$ is not 1 on a neighborhood of $K$, this discrepancy may be neglected: it is straightforward to modify $\varphi$ near $K$ so that it is 1 on a neighborhood of $K$ and such that the Dirichlet energy changes by an arbitrarily small amount.)  From this, we can verify that the capacity of $K$ in $(M,g)$ equals $\frac{2}{C}$:
\begin{align*}
\capac(K) &= \frac{1}{4\pi} \int_M |\nabla \varphi|^2 dV\\
&=\frac{1}{4\pi} \int_{-\infty}^\infty \frac{1}{C^2}f(s)^{-4} (4\pi f(s)^2) ds\\
&= \frac{2}{C}.
\end{align*}

Now, consider a sequence of smooth, positive functions $f_i :[-2i, \infty) \to \R$ such that $f_i(s) = f(s)$ for $s \geq -i $, and such that $g_i = ds^2 + f(s)^2 d\sigma^2$ is a smooth Riemannian metric with a pole at $s=-2i$, i.e. $f_i(-2i)=0$, so the underlying manifold is diffeomorphic to $\R^3$.

Then for every $i$, the capacity of $K$ with respect to $g_i$ equals $1/C$, i.e. is half the capacity of $K$ in $(M,g)$. This can be seen by observing the capacity of $K$ in $(M,g_i)$ is achieved by the function that is 1 for $-2i\leq s\leq 0$ and otherwise agreeing with $\varphi$ above.

But the $g_i$ converge smoothly on compact sets to $g$, and the set $K$ has capacity $2/C$ in the limit space.

\medskip

\paragraph{\emph{Example 3: Capacity jump with $\VF$-convergence}}

Let $Y$ be Euclidean 4-space with coordinates $(x,y,z, w)$, and let $X$ be the $w=0$ subspace. Then $X$ naturally becomes a local integral current space $N$ of dimension 3 with the Euclidean metric, where the locally integral current is given by integration, oriented up (i.e., in the $+w$ direction). Obviously $X$ is isometrically embedded in $Y$.

Let $K =\{ (x,y,z, 0) \; | \; x^2+y^2+z^2 \leq 1\}$. For each $i=1,2,\ldots$, define
$$X_i = K \cup \{(x,y,z, 1/i) \; : \; x^2+y^2+z^2 \geq 1\}.$$
Letting $X_i$ have the induced Euclidean metric, $X_i$ is obviously isometrically embedded in $Y$. $X_i$ may also be equipped with the locally integral 3-current given by integration, oriented up, producing a local integral current space, $N_i$. See Figure \ref{fig3}.

\begin{figure}
    \centering
    \includegraphics[scale=0.8]{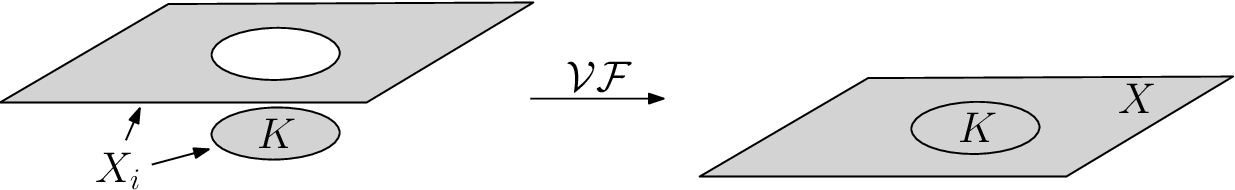}
    \caption{In Example 3, the space $X_i$ is the union of a unit 3-ball $K$ and a 3-space minus a unit ball, sitting at height $\frac{1}{i}$ above $K$ in $\R^4$. $X_i$, naturally viewed as a local integral current space, pointed $\VF$-converges to $X$  with respect to the origin, which is Euclidean 3-space. The corresponding regions for $K$ are simply $K_i=K$. 
    }
    \label{fig3}
\end{figure}

Observe that $N_i$ converges to $N$ in the pointed $\VF$-sense as $i \to \infty$, where all the points are chosen to be the origin in $Y$. This can easily be seen from the fact that $X_i$ and $X$ are isometrically embedded in $Y$, and $X_i \to X$ in the usual local flat sense in $Y$.

Let $u$ be the defining function for $K$ in $X$ given as the signed distance in $X$ to $\partial K$, negative inside of $K$, and let $U: Y \to \R$ be the standard Lipschitz extension. Consider $K_i = U^{-1}(-\infty,0] \cap X_i$, a sequence of corresponding regions as in Section \ref{sec_main}. We claim $K_i=K$. If $x \in K$, then $u(x)=U(x) \leq 0$. Since $K \subset X_i$, we have $x \in K_i$. On the other hand, suppose $p \in K_i$, so $U(p) \leq 0$. Then there exists $x \in X$ such that
$$u(x) + d_Y(x,p) \leq 0.$$
Clearly $u(x) \leq 0$, i.e. $x \in K$. We can see the defining function is given by $u(x) = d_Y(0,x) -1$. With the triangle inequality, we have
$$d_Y(p,0) \leq 1,$$
i.e., $x$ belongs to the closed unit ball in $Y$ about $p$. The latter only intersects $X_i$ at $K$, so $p \in K$.

Now, the capacity of $K$ in $X$ is positive, but the capacity of $K_i$ in $X_i$ is zero for all $i$. This is easy to see because $X_i$ is disconnected: the function that equals 1 on $K$ and vanishes on $X_i \setminus K$ is Lipschitz and is a valid test function for the capacity, with zero Dirichlet energy. Thus, in \eqref{eqn_limsup_sublevel} of Theorem \ref{thm_sublevel}, we have strict inequality (without needing to take a subsequence).

If desired, one can arrange a similar example with the $X_i$ connected, as follows. Join the two connected components of $X_i$ with a thin solid tube of 3-volume of $O(1/i^3)$ and length $O(1/i)$. Then with a Lipschitz test function $f_i$ equalling 1 on $K$, with $\Lip(f_i)$ of $O(i)$ on the tube, and 0 elsewhere, the Dirichlet energy of $f_i$ would be $O(1/i)$, i.e., the capacity of $K$ in the connected space $X_i$ would still converge to 0.

\medskip 

\paragraph{\emph{Example 4: Cancellation and necessity of volume-preserving $\F$ convergence}}
Here, we demonstrate that upper semicontinuity of capacity may fail for pointed $\F$-convergence, without assuming $\VF$-convergence. We exploit the ``cancellation'' phenomenon of intrinsic flat convergence as in \cite[Example A.19]{SW}.

Let $Y$ be Euclidean 4-space as in the previous example, and let $X_i$ be the union of the $w=0$ hyperplane and an annular region sitting slightly above:
$$X_i = \{(x,y,z,w) \; | \; w=0\} \cup \{(x,y,z,1/i) \; | \; 1 \leq x^2 + y^2+z^2 \leq 4\}.$$
Equip $X_i$ with the induced metric, so that $X_i$ is isometrically embedded in $Y$. Let $T_i$ be the locally integral current on $X_i$ given by integration, oriented up on the $w=0$ hyperplane and down on the annular region. $X_i$ with the induced metric, equipped with $T_i$, produces a sequence of local integral current spaces, $N_i$. Letting $K$ be the unit ball $$\{(x,y,z, 0) \; | \;x^2 + y^2 +z^2 \leq 1\},$$
we have $K\subset X_i$, and the capacity of $K$ in $X_i$ is a positive constant independent of $i$.

Now, $N_i$ converges in the pointed $\F$-sense (but not $\VF$-sense) to 
$$X=K \cup \{(x,y,z,0) \; | \; x^2+y^2+z^2 \geq 4\},$$
with the induced metric and the integral current given by integration, oriented up. See Figure \ref{fig4}. Here, all the base points are chosen to be the origin. Since $K$ is a compact component of $X$, we have $\capac_N(K)=0$. Using $r=1$, we have a violation of Theorem \ref{thm_balls} if $\VF$-convergence is not assumed.

\begin{figure}
    \centering
    \includegraphics[scale=0.8]{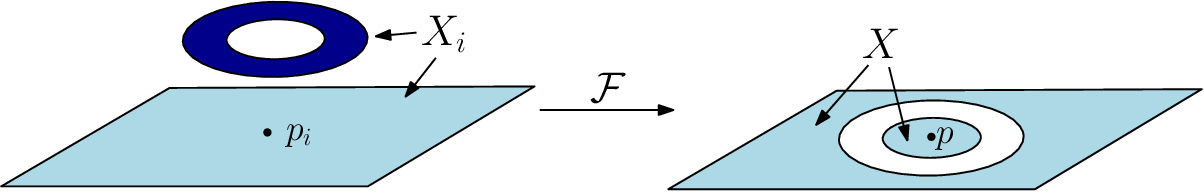}
    \caption{In Example 4, the space $X_i$ is the union of a hyperplane with an oppositely-oriented annular region sitting above at height $\frac{1}{i}$ in $\R^4$. $X_i$, naturally viewed as a local integral current space, converges in the pointed $\F$-sense (but not $\VF$) to $X$, which is Euclidean 3-space minus an annular region representing where the cancellation occurred.
    }
    \label{fig4}
\end{figure}

\section{Asymptotically flat local integral current spaces and general relativistic  mass}
\label{sec_mass}

Asymptotically flat (AF) Riemannian manifolds are of particular interest in the study of general relativity. These spaces are characterized by their metric tensors (and derivatives) decaying in a precise sense to the Euclidean metric in some appropriate coordinate chart that covers all but a compact set. The ADM mass is a numerical geometric invariant of an AF manifold that is of both significant physical and geometric interest \cite{ADM}. As described in the introduction, a number of open problems seem to necessitate an understanding of asymptotic flatness and ADM mass for spaces that are neither smooth nor Riemannian (again, we refer the reader to  \cite{Sor2} and \cite{JL}, for example).

In this section, we give a possible definition of asymptotic flatness for local integral current spaces and describe two possible definitions of general relativistic  mass for such spaces.
\medskip

We begin with a generalization of asymptotic flatness to metric spaces:
\begin{definition}
We define a metric space $(X,d)$ to be \emph{asymptotically flat of dimension $n \geq 3$} if for any $\epsilon >0$, there exists a compact set $K \subset X$ and a bijective map $\Phi$ from $X \setminus K$ to $\R^n \setminus B$ (for a closed ball  $B \subset \R^n$) that is bi-Lipschitz when $X \setminus K$ and $\R^n \setminus B$ are endowed with the restricted distance of $d$ and of the Euclidean distance function, respectively, such that
$$\Lip(\Phi),\Lip(\Phi^{-1}) \leq 1+\epsilon.$$
\end{definition}

It is possible to show that any AF Riemannian manifold of dimension $n$ (in the usual sense) is an AF metric space of dimension $n$ with its natural distance function.

Similarly, we define a metric measure space $(X,d,\mu)$ to be AF of dimension $n \geq 3$ if the above properties hold for $(X,d)$ and also if
$$(1+\epsilon)^{-n}\mathcal{L}^n  \leq \Phi_{\#}(\mu) \leq (1+\epsilon)^n\mathcal{L}^n$$
as Borel measures on $\R^n \setminus B$, where $\mathcal{L}^n$ is the Lebesgue measure. For example, if $(X,d)$ is an AF metric space of dimension $n$, then equipped with Hausdorff $n$-measure, it becomes an AF metric measure space.

Now we can define a local integral current space  $(X,d,T)$ of dimension $n$ to be AF if $(X,d, \|T\|)$ is an asymptotically flat metric measure space of dimension $n$. (We note other reasonable definitions are possible, and that the current definition does not involve orientations or tangent cones.)  In this setting, the capacity of compact sets is well defined, as is the boundary mass of balls for almost all radii \cite[Lemma 2.34]{Sor}.

\medskip

We now proceed to discuss the concept of general relativistic mass for asymptotically flat local integral current spaces (not to be confused with the mass measure). The standard definition of ADM mass involves derivatives of the Riemannian metric coefficients and so is  unsuitable for metric spaces.

A well-known approach to a ``weak'' understanding of ADM mass is due to Huisken \cites{Hui1,Hui2}: his so-called \emph{isoperimetric mass} uses only volumes and areas (perimeters) in its formulation. In dimension three, with nonnegative scalar curvature, it is known to equal the ADM mass in the smooth asymptotically flat case \cites{Hui1,Hui2,JLC0,CESY}. In \cite{JL}, Jauregui and Lee gave a definition of asymptotically flat local integral current space (more restrictive than that which we use here, essentially requiring the complement of a compact set to be a smooth manifold with a $C^0$ Riemannian metric), and used Huisken's isoperimetric mass as a substitute for ADM mass. Since the perimeters of compact sets are well defined even for $C^0$ Riemannian metrics, it was clear that Huisken's definition was well defined. 

Huisken's isoperimetric mass, $m_{iso}$, is typically defined for asymptotically flat Riemannian 3-manifolds. We can generalize this concept to any 3-dimensional asymptotically flat local integral current space, using boundary mass in place of perimeter: 
$$m_{iso}(X,d,T)=\sup_{\{K_j\}} \limsup_{j \to \infty}  \frac{2}{\M(\partial(T \llcorner K_j))} \left[\|T\|(K_j)-\frac{1}{6 \sqrt{\pi}} \M(\partial(T \llcorner K_j))^{\frac{3}{2}}\right]\qquad \in [-\infty,\infty],$$
 where $\{K_j\}$ is an exhaustion of $X$ by compact sets. Note that if $\M(\partial(T \llcorner K_j)) = \infty$, the expression inside the $\limsup$ is $-\infty$. In particular, we may restrict to exhaustions such that $\M(\partial(T \llcorner K_j))$ is finite, which is equivalent to saying $T \llcorner K_j$ is an integral $n$-current on $(X,d)$. 

Huisken's definition was inspired by the isoperimetric inequality: far out in the AF end, the inequality almost holds, and the ADM mass can be detected through the deficit. Jauregui proposed a corresponding definition of mass based on the isocapacitary inequality (that the capacity of a compact set of a given volume in $\R^n$  is minimized by balls) \cite{Jau}. This definition of ``capacity-volume mass'' was for AF manifolds, including $C^0$ AF manifolds. However, it can be generalized to AF local integral current spaces of dimension $3$ as follows:
$$m_{CV}(X,d,T) =  \sup_{\{K_j\}} \limsup_{j \to \infty} \frac{1}{4\pi\capac(K_j)^2 }  \left[ \|T\|(K_j) - \frac{4\pi}{3} \capac(K_j)^{3}\right],$$
where the capacity is defined as in \eqref{def_cap}.
In \cite{Jau}, strong evidence was given for $m_{CV}$ recovering the ADM mass in the smooth case with nonnegative scalar curvature (and hence serving as a weak stand-in for the ADM mass). Furthermore, it was observed that capacity is in some ways better behaved than perimeter or boundary mass --- for example, capacity is less sensitive to perturbations, and as confirmed by our main theorems and discussed below, has a favorable semicontinuity property --- so $m_{CV}$ may ultimately be easier to work with in low-regularity ADM mass problems than $m_{iso}$.

To connect this study of mass with our main theorems, we conclude with a discussion of the lower semicontinuity of total mass in general relativity. In \cites{JC2,JLC0} it was shown that the ADM mass functional (and more generally, Huisken's isoperimetric mass)  is lower semicontinuous on an appropriate class of asymptotically flat 3-manifolds of nonnegative scalar curvature, for pointed $C^2$, and more generally, for pointed $C^0$ Cheeger--Gromov convergence. This was further generalized to pointed $\VF$ convergence, under natural hypotheses, using Huisken's isoperimetric mass as a stand-in for the mass of the (potentially non-smooth) limit space \cite{JL}. Below, we argue that Theorem \ref{thm_sublevel} supports lower semicontinuity of $m_{CV}$ in dimension three.

To simplify the discussion, we recall from the appendix of \cite{Jau} that $m_{CV}$ may alternatively be written:
$$m_{CV}(X,d,T) = \sup_{\{K_j\}} \limsup_{j \to \infty} \left[ \left(\frac{3\|T\|(K_j)}{4\pi }\right)^{1/3} - \capac(K_j) \right].$$

Now consider the inner expression $\left(\frac{3\|T\|(K)}{4\pi }\right)^{1/3} - \capac(K)$ as a functional on compact sets $K$. To have any hope of showing $m_{CV}$ is lower semicontinuous under pointed $\VF$ convergence, it seems necessary to know that ``volume radius minus capacity'' itself is lower semicontinuous. Since volume is by definition continuous in $\VF$, this amounts to the statement that capacity is \emph{upper} semicontinuous. We demonstrated this in Theorem \ref{thm_sublevel} for example. In other words, the results of this paper are supportive of $m_{CV}$ itself being lower semicontinuous under pointed $\VF$-convergence, though a full proof of this is more subtle, requiring, for example, an unproven analog of the ADM mass estimate \cite[Theorem 17]{JLC0}, but for the capacity-volume mass in place of the isoperimetric mass.

\section*{Appendix A: capacity under pointed $C^0$ Cheeger--Gromov convergence}
In this brief appendix we verify the claim stated in the introduction that upper semicontinuity holds for the capacity in Riemannian manifolds converging in the pointed $C^0$ Cheeger--Gromov sense. A precise statement of this fact requires a much simpler analog of the ``corresponding regions'' $K_i$ used in Section \ref{sec_main}:

\begin{prop}
Let $(M_i, g_i, p_i)$ be a sequence of complete pointed $C^0$ Riemannian $n$-manifolds, $n \geq 3$, converging in the pointed $C^0$ Cheeger--Gromov sense to a complete pointed  $C^0$ Riemannian $n$-manifold $(N,h,q)$ (we recall the definition in the proof below). Let $K \subseteq N$ be a compact set. Then for $i$ sufficiently large, there exist compact sets $K_i \subseteq N_i$ and diffeomorphisms $\Phi_i$ mapping a neighborhood $\Omega$ of $K$ in $M$ onto a neighborhood of $K_i$ in $M_i$ with $\Phi_i(K)=K_i$ such that $\Phi_i^* g_i$ converges uniformly to $g$ on $\Omega$. Moreover,
$$\limsup_{i \to \infty} \capac_{N_i}(K_i) \leq \capac_N(K).$$
\end{prop}

\begin{proof}
Pointed $C^0$ Cheeger--Gromov convergence means that for every $r>0$ there exists a domain $\Omega \supseteq B(q,r)$ in $(N,h)$ and, for $i$ sufficiently large, smooth embeddings $\Phi_i: \Omega \to M_i$ such that $\Phi_i(\Omega)$ contains the open ball of radius $r$ about $p_i$ in $(M_i,g_i)$, and that the metrics $\Phi_i^* g_i$ converge uniformly to $h$ on $\Omega$.

Given $\epsilon > 0$, let $\phi$ be a valid test function for $\capac_N(K)$, i.e., $\phi$ is Lipschitz, with compact support in $N$, and $\phi \equiv 1$ on $K$. By the definition of capacity, we may assume
$$\frac{1}{(n-2) \omega_{n-1}} \int_N |\nabla \phi|^2_h dV_h \leq \capac_N(K) + \epsilon,$$
where, in this proof, we place subscripts on the gradient norm and volume form to indicate the Riemannian metric being used.  Without loss of generality, we assume $\phi$ is smooth.

Choose $r>0$ sufficiently large so that $B(q,r)$ contains the support of $\phi$ (which certainly contains $K$). As in the definition of pointed Cheeger--Gromov convergence, take an appropriate domain $\Omega$ and appropriate smooth embeddings $\Phi_i: \Omega \to M_i$ (for $i$ sufficiently large). Let $\phi_i : M_i \to \R$ be defined by $\phi \circ \Phi_i^{-1}$ on $\Phi_i(\Omega)$, extended to $M_i$ by zero. It is easy to see $\phi_i$ has compact support in $\Phi_i(\Omega)$ and is smooth.

Since $\Omega \supset K$, we may define $K_i = \Phi_i(K)$. Note that $\phi_i$ is identically 1 on $K_i$, so that $\phi_i$ is a valid test function for the capacity of $K_i$.
Now
\begin{align*}
\capac_{M_i}(K_i) &\leq \frac{1}{(n-2)\omega_{n-1}} \int_{N_i} |\nabla \phi_i|^2_{g_i} dV_{g_i}\\
&= \frac{1}{(n-2)\omega_{n-1}} \int_{\Phi_i(\Omega)} |\nabla \phi_i|^2_{g_i} dV_{g_i}\\
&= \frac{1}{(n-2)\omega_{n-1}} \int_{\Omega} |\nabla \phi|^2_{\Phi_i^*g_i} dV_{\Phi_i^* g_i}.
\end{align*}
Taking the $\limsup$ and using the fact that $\Phi_i^* g_i$ converges uniformly to $h$ on $\Omega$, we obtain:
\begin{align*}
\limsup_{i \to \infty}\capac_{M_i}(K_i)
&\leq \frac{1}{(n-2)\omega_{n-1}} \int_{\Omega} |\nabla \phi|^2_{h} dV_{h}\\
&= \frac{1}{(n-2)\omega_{n-1}} \int_{M} |\nabla \phi|^2_{h} dV_{h}\\
&\leq \capac_N(K) + \epsilon.
\end{align*}
Since $\epsilon > 0$ was arbitrary, the claim follows.
\end{proof}

\section*{Appendix B: tangential differential, Dirichlet energy, Sobolev spaces, and capacity}
In this section we first review the definition of  Dirichlet energy for a Lipschitz function defined on the canonical set of a current as was done by Portegies \cite{Por}, which we used in the definition of capacity. This will require the concepts of metric and $w^*$-differentials, approximate tangent spaces, and the tangential differential. After that, we relate the latter to the minimal relaxed gradient and briefly discuss several notions of Sobolev spaces on metric spaces. We conclude with a comparison of the definition of capacity we employ in this paper and other definitions appearing in the literature.

A \emph{$w^*$-separable Banach space} $Z$ is by definition a dual space $Z = G^*$, of a separable Banach space $G$, and hence $Z$ is Banach with norm $\| \cdot \|$. Let $\{g_j\}_{j=0}^\infty$ be a countable dense subset in the unit ball in $G$. The function $d_w: Z \times Z \to \R$ given by  
\begin{equation*}\label{eq-distance-d_w}
d_w(x,y) := \sum_{j=0}^\infty 2^{-j} | \langle x-y, g_j \rangle|, 
\quad \text{for } x,y \in Z,
\end{equation*}
is a distance. Note that any such $d_w$ induces the $w^*$-topology on bounded subsets of $Z$, and $(Z,d_w)$ is a separable space \cite[Section~2]{AK_rect}.
One of the main examples of a $w^*$-separable Banach space is the space $\ell^\infty= (\ell^1)^*$. 

\begin{definition}[{\cite[Definitions 3.1 and 3.4]{AK_rect}}]
Let $Z$ be a metric space and $g:\R^n \to Z$ a function. 
\begin{itemize}
\item We say that $g$ is \emph{metrically differentiable} at $x \in \R^n$ if there is a seminorm $md_xg : \R^n \to \R$ such that
\begin{equation*}
d(g(y), g(x)) - md_x g (y - x) = o(|y - x|),  \quad  y \to x.
\end{equation*}
 We call $md_x g$ the \emph{metric differential} of $g$ at $x$.
 \item If $Z$ is a $w^*$-separable Banach space, we say that $g$ is \emph{$w^*$-differentiable} at $x\in\R^n$ if there is a linear map $wd_xg:\R^n \to Z$ such that
\begin{equation*}
\lim_{y\to x} \frac{g(y) - g(x) - wd_xg (y-x)}{|y-x|} = 0,
\end{equation*}
where the limit is understood in the $w^*$-sense. 
The map $wd_x g$ is called the \emph{$w^*$-differential} of $g$ at $x$.
\end{itemize}
 \end{definition}

For Lipschitz maps the following is known. 

\begin{thm}[{\cite[Theorems 3.2 and  3.5]{AK_rect}}]
If $Z$ is a metric space, then any Lipschitz function $g: \R^n \to Z$ 
is metrically differentiable  $\lebmeas^n$-a.e.
If additionally, $Z$ is a $w^*$-separable Banach space, then $g$
is also $w^*$-differentiable $\lebmeas^n$-a.e, and the metric and weak differential satisfy 
\begin{equation*}
md_xg(v) = \|wd_xg(v)\|, \qquad \text{ for all } v \in \R^n \text{ and }
\lebmeas^n\text{-}a.e. \, x \in \R^n.
\end{equation*}
\end{thm}
\medskip 

A subset $S$ of a metric space $Z$ is \emph{countably $\cH^n$-rectifiable} if there exist Lipschitz functions $g_j: A_j \subset \R^n \to Z$, $j \in \mathbb N$, defined on Borel sets $A_j$ such that 
\begin{equation*}
\cH^n \left( S \backslash \bigcup_{j=1}^\infty g_j(A_j) \right) = 0.
\end{equation*}
If $Z$ is a $w^*$-separable Banach space, the \emph{approximate tangent space} to $S \subseteq Z$ at a point $x$ is defined as 
\begin{equation*}
\Tan(S,x) = wd_y g_j(\R^n),
\end{equation*} 
whenever  $y = g_j^{-1}(x)$ and $g_j$ is metrically and $w^*$-differentiable at $y$, with $J_n(wd_y g_j) > 0$, where for any linear function $L: V \to W$ between two Banach spaces, with $n=\dim V$, 
$$J_n(L)= \frac{\omega_n} {\cH^n\{v \in V\, :\, ||L(v)|| \leq 1\}}$$
denotes the $n$-Jacobian of $L$. By \cite{AK_rect}, $\Tan(S,x)$ is well defined for $\cH^n$-almost all $x \in S$.  A finite Borel measure $\mu$ is called \emph{$n$-rectifiable} if $\mu = \zeta \cH^n \llcorner S$ for a countably $\cH^n$-rectifiable set $S$ and a Borel function $\zeta:S \to (0,\infty)$.

The next theorem shows the existence of tangential differentials of Lipschitz functions on rectifiable sets.

\begin{thm}[{\cite[Theorem 8.1]{AK_rect}}]
Let $Z$ and $Z'$ be two $w^*$-separable Banach spaces, $S \subset Z$ an $\cH^n$-countably rectifiable subset and $f: Z \to Z'$ a Lipschitz function. Let $\zeta:S \to (0,\infty)$
be an $ \cH^n $-integrable function and denote by $\mu = \zeta \cH^n \llcorner S$ the corresponding $n$-rectifiable measure.

Then for $\cH^n$-almost every $x\in S$, there exist a Borel set $S^x \subset S$ such that
the upper $n$-dimensional density of 
$\mu \llcorner S^x$ equals zero, 
\begin{equation*}
  \Theta_n^*( \mu \llcorner S^x, x) = 0,  
\end{equation*}
and a linear and $w^*$-continuous map $L: Z \to Z'$ so that
\begin{equation*}
\lim_{y \in S \backslash S^x \to x} \frac{d_w(f(y),f(x) + L(y-x))}{|y - x|} = 0.
\end{equation*}
$\Tan(S,x)$ exists and $L$ is uniquely determined on $\Tan(S,x)$ and its restriction to $\Tan(S,x)$
is called the tangential differential to $S$ at $x$ and is denoted by
\begin{equation*}
d_x^S f: \Tan(S,x) \to Z'.
\end{equation*}
Furthermore, the tangential differential is characterized by the property that for any Lipschitz map $g: A \subset \R^n \to S$,
\begin{equation*}
wd_y(f \circ g) = d_{g(y)}^S f \circ wd_y g, \qquad \text{for } \mathcal{L}^n\text{-a.e. } y \in A.
\end{equation*}
\end{thm}
Note that if $d_x^S f$ is defined, then its dual norm satisfies
\begin{equation}
|d_x^S f| \leq \Lip(f).
\end{equation}

For an arbitrary countably $\cH^n$-rectifiable metric space $S$, the fact that we can always embed it into a $w^*$-separable dual space allows us to define approximate tangent spaces to $S$ in the following way.
\begin{definition}[{\cite[Definition~5.9]{AK_rect}}]
Let $S$ be a separable, countably $\cH^n$-rectifiable metric space, let $Z$ be a $w^*$-separable Banach space and let $\iota : S \to Z$ be a distance-preserving map. Then for $\cH^n$-almost every $x \in S$ the \emph{approximate tangent space} of $S$ at $x$ is defined as 
\begin{equation}
\Tan(S,x) : = \Tan(\iota(S),\iota(x)).
\end{equation}
\end{definition}

Even though the definition seems to depend on the choice of a distance-preserving map, the approximate tangent space $\Tan(S,x)$ is actually uniquely determined $\cH^n$-a.e.~up to linear isometries \cite{AK_rect}. 

\begin{definition}[{\cite[Equation~(3.1)]{Por}}]
Let $X$ be a separable metric space, let $S \subset X$ be an $\cH^n$-countably rectifiable subset, let $\zeta : S \to (0, \infty)$ be an $\cH^n$-integrable function and let $\mu = \zeta \cH^n \llcorner S$ be the corresponding $n$-rectifiable measure. Let $f : S \to \mathbb{R}$ be a Lipschitz function.  Let $Z$ be a $w^*$-separable Banach space, and let $\iota : S \to Z$ be a distance-preserving map. Then we define
\begin{equation*}
|d_x^S f| := |d_{\iota(x)}^{\iota(S)} (f\circ \iota^{-1})|,
\end{equation*}
for $\mu$-a.e.~$x \in S$, where the right hand side denotes the dual norm of $d_{\iota(x)}^{\iota(S)} (f\circ \iota^{-1})$.
\end{definition}
This quantity is also well defined, independent of the isometric embedding \cite{Por}.

\begin{definition}[Definition 3.8 of \cite{Por}]\label{de:NormEn}
Let $X$ be a complete metric space, and let $T \in \In(X)$. Let $S = \set(T)$, and let $f: S \to \R$ be a Lipschitz function. Then the \emph{(Dirichlet) energy} of $f$ is given by
\begin{equation*}
E_T(f) := \int_X |d_x^S f |^2 \, d\|T\|(x).
\end{equation*}
\end{definition}

The energy of $f$ is invariant under distance-preserving maps \cite[Proposition~3.1]{Por}, and for any compact oriented Riemannian manifold $(M,g)$ we have that the energy is given by $\int_M |\nabla f|^2dV$, where the gradient and volume measure are taken with respect to $g$. 

We will next mention some results that essentially state that $|d^S_x f|$ and its associated Dirichlet energy correspond to the minimal relaxed gradient and the Cheeger energy as introduced in \cite{AGS}. One could therefore have used the latter definitions as a starting point, but the characterization of the minimal relaxed gradient as the norm of the tangential derivative is important in the main text.

Let $X$ be a complete metric space, $T \in \In(X)$ and $S=\set(T)$. 
The space $L^2(\|T\|)$ is the Hilbert space of equivalence classes of functions on $X$ that are square-integrable with respect to $\|T\|$ with inner product
\begin{equation*}
\langle f, g\rangle_{L^2(\|T\|)} := \int_X f g \,  d\|T\|.
\end{equation*}
We define a Sobolev space $W^{1,2}(\|T\|)$ as follows.

\begin{definition}[{\cite[Section~4]{Por}}]
For an integral current $T$ on a complete metric space $X$, the space $W^{1,2}(\|T\|)$ is defined as the completion of the set of bounded Lipschitz functions on $\spt T$ with respect to the norm $\|.\|_{W^{1,2}}$ given by
\begin{equation*}
\begin{split}
\|f \|_{W^{1,2}}^2 
& = \int_X f^2 d\|T\| + \int_X |d^S_x f|^2 d\|T\|(x). 
\end{split}
\end{equation*}
\end{definition}
By this definition, every $f$ in $W^{1,2}(\|T\|)$  can be represented by a Cauchy sequence $f_i$ of bounded Lipschitz functions.

\begin{definition}[{\cite[Equation~(4.2)~and~Definition 4.1]{Por}}]
We denote by $\mathcal{T}^*_2(\|T\|)$  where $S$ denotes $\set(T)$, the Banach space of equivalence classes of covector fields endowed with the norm
\begin{equation*}
\| \psi \|_{\mathcal{T}^*_2(\|T\|)}^2 := \int_X | \psi(x) |^2_{[\Tan(S,x)]^*} d\|T\|(x),
\end{equation*}
If $(f_i)$ is a sequence of bounded Lipschitz functions representing $f \in W^{1,2}(\|T \|)$, then the limit of the sequence $(d^S_xf_i)$ in $\mathcal{T}^*_2(\|T\|)$
is denoted by $d_xf$.
\end{definition}

\begin{thm}[Theorem 5.2 of \cite{Por}]
\label{thm_df}
Let $X$ be a complete metric space, $T \in \In(X)$ and 
$f \in L^2(\|T\|)$. Then $f$ has a relaxed gradient in the sense of  \cite[Definition 4.2]{AGS} if and only if $f \in W^{1,2}(||T||)$. Moreover, the minimal relaxed gradient equals $|d_x f|$ for $||T||$-a.e. $x \in X$.
\end{thm}

As a consequence, the space $W^{1,2}(\|T\|)$ agrees with the Sobolev space $W^{1,2}(X, d, \|T\|)$ as introduced in \cite[Remark~4.6]{AGS}, and the Dirichlet energy in Definition \ref{de:NormEn} corresponds up to a factor $2$ to the Cheeger energy defined in \cite[Theorem~4.5]{AGS}.

We conclude with a discussion of Sobolev spaces and capacity. 
In  \cites{BB, HKST} metric measure spaces $(X,d,m)$, where $(X,d)$ is separable and 
$m$ is a locally finite Borel regular measure on $X$,
are considered and Newtonian
spaces of functions  $N^{1,p}(X,d,m)$, $1 \leq p < \infty$, are defined. Originally defined by Shanmugalingam in her PhD thesis and subsequent paper \cite{Sha}, these are a type of Sobolev space.
Then the \emph{$p$-Sobolev capacity} of a set $E \subset X$ 
is defined as 
\begin{align*}
\capac_p(E)&= \inf\left\{  \int |u|^pdm + \int \rho_u^p  dm \,\right.\\
&\qquad\qquad\left.: \, u \in N^{1,p}(X,d,m), u \geq 1 \text{ on } E \text{ outside a $p$-exceptional set of measure zero}\right\},
    \end{align*}
    where $\rho_u$ denotes the minimal $p$-weak upper gradient of $u$. 
It is also shown that this is equivalent to  \cite[Lemma 7.2.6]{HKST}:
\begin{equation*}
\capac_p(E)= \inf\left\{  \int |u|^pdm + \int \rho_u^p  dm \, : \, u \in N^{1,p}(X,d,m), \, 0 \leq u \leq 1 \text{ and $u=1$ on $E$}\right\}.
    \end{equation*}
    
    Other types of Sobolev spaces on metric spaces have been defined, and Theorem 10.5.1---10.5.3 in \cite{HKST} describe some relations between them.  For $1<p < \infty$, the Cheeger space,
   $W_{Ch}^{1,p}$, and Newtonian space, $N^{1,p}$, are equal (up to representatives) and both norms coincide \cite[Theorem~10.5.1]{HKST}.
    Provided 
$m$ is a doubling measure and $X$ satisfies a $q$-Poincar\'e inequality, $1 \leq q <2$, 
the following Sobolev spaces coincide 
\begin{equation}
M^{1,2}=P^{1,2}=KS^{1,2}=N^{1,2}=W_{Ch}^{1,2},
\end{equation}
and the norms $\|\cdot \|_{M^{1,2}}$, $\| \cdot \|_{N^{1,2}}=\|\cdot\|_{W_{Ch}^{1,2}}$ and $\| \cdot \|_{KS^{1,2}}$ are comparable \cite[Theorem~10.5.3]{HKST}. Here $M^{1,2}$ is the Haj\l asz Sobolev space \cite{Haj}, $P^{1,2}$ is  the Poincar\'e Sobolev space, and $KS^{1,2}$
is the Korevaar--Schoen Sobolev space. 
If $m$ is only a doubling measure then \cite[Theorem~10.5.1]{HKST}
\begin{equation}
    M^{1,2} \subseteq P^{1,2} \subseteq KS^{1,2} \subseteq N^{1,2}=W^{1,2}_{Ch} .
\end{equation}
        
For complete and separable metric measure spaces $(X,d,m)$, in \cite[Theorem 6.2]{AGS} (cf. \cite[Theorem 2.2.28]{GP})
it was shown that $W^{1,2}_{Ch}$ and the Sobolev space $W^{1,2}(X,d,m)$ using the minimal relaxed gradient in the sense of  \cite[Definition 4.2]{AGS} are the same, and their norms coincide. Furthermore, any $f \in W^{1,2}(X,d,m)$ can be approximated by functions in $\Lip(X) \cap L^2(m)$.

The main difference between the  capacity we use in this paper and the 2-Sobolev capacity given in \cite{HKST}, is that in our definition of capacity
we only integrate the gradient term.

For a metric measure space $(X,d,m)$ and $\Omega \subset X$, the $p$-variational capacity
is defined as
\[
\capac_p^{\circ}(E, \Omega) = \inf\left\{  \int \rho_u^p  dm \, : \, u \in N_0^{1,p}(\Omega,d,m), \, 0 \leq u \leq 1 \text{ and $u=1$ on $E$}\right\}.
\]
If $(X,d,m)$ is doubling and admits a Poincar\'e inequality, then certain inequalities hold between the variational capacity and the Sobolev capacity
(see  \cite[Theorem 6.16]{BB}).

\end{document}